\documentclass[11pt,oneside,a4paper]{article}
\usepackage[left=2.5cm,right=2.5cm,top=2.7cm,bottom=2.7cm]{geometry}

\usepackage{mathptmx}
\usepackage{times}
\usepackage{amsmath}
\usepackage{amsthm}
\usepackage{amssymb}
\usepackage[all]{xy}
\usepackage[latin1]{inputenc}
\usepackage{graphicx}
\usepackage[usenames,dvipsnames]{color}
\usepackage{verbatim}
\usepackage{flafter}
\usepackage[pdffitwindow=false, pdfview=FitH,pdfstartview=FitH,plainpages=false]{hyperref}
\usepackage{hyperref}
\usepackage{tocloft}
\usepackage{multirow}

\DeclareFontFamily{U}{wncy}{}
\DeclareFontShape{U}{wncy}{m}{n}{<->wncyr10}{}
\DeclareSymbolFont{mcy}{U}{wncy}{m}{n}
\DeclareMathSymbol{\Sh}{\mathord}{mcy}{"58}

\newtheorem{theorem}{Theorem}[section]
\newtheorem{lemma}[theorem]{Lemma}
\newtheorem{corollary}[theorem]{Corollary}

\newtheorem{proposition}[theorem]{Proposition}
\theoremstyle{definition}

\newtheorem{definition}[theorem]{Definition}

\newtheorem{remark}[theorem]{Remark}

\newcommand{\Sel}{\mathrm{Sel}}
\newcommand{\CH}{\mathrm{CH}}
\renewcommand{\H}{\mathrm{H}}
\newcommand{\CM}{\mathrm{CM}}
\newcommand{\NS}{\mathrm{NS}}

\newcommand{\rank}{\mathrm{rank}}
\newcommand{\Gal}{\mathrm{Gal}}

\newcommand{\Pic}{\mathrm{Pic}}

\newcommand{\ord}{\mathrm{ord}}
\newcommand{\cor}{\mathrm{cor}} 
\newcommand{\Corr}{\mathrm{Corr}}
\newcommand{\Tr}{\mathrm{Tr}} 
\newcommand{\Hom}{\mathrm{Hom}} 
\newcommand{\Frob}{\mathrm{Frob}}

\newcommand{\End}{\mathrm{End}}

\newcommand{\GL}{\mathrm{GL}}
\newcommand{\Sym}{\mathrm{Sym}}
\newcommand{\Spec}{\mathrm{Spec}}

\newcommand{\res}{\mathrm{res}}
\newcommand{\et}{\mathrm{et}}
\newcommand{\ur}{\mathrm{ur}}

\renewcommand{\exp}{\mathrm{exp}}

\newcommand{\QQ}{\mathbb{Q}}
\newcommand{\RR}{\mathbb{R}}
\newcommand{\CC}{\mathbb{C}}
\newcommand{\ZZ}{\mathbb{Z}}

\newcommand{\OO}{\mathcal{O}}

\newcommand{\M}{\mathrm{M}}
\newcommand{\SL}{\mathrm{SL}}
\newcommand{\SO}{\mathrm{SO}}
\newcommand{\PGL}{\mathrm{PGL}}

\begin{document}

\title{CM cycles on Kuga--Sato varieties over Shimura curves and Selmer groups}
\maketitle
\begin{center}
\author{\textbf{Yara Elias and Carlos de Vera-Piquero}} \\
\textit{Max-Planck Institute for Mathematics, Bonn, Germany\\
\emph{elias@mpim-bonn.mpg.de}}\\
\textit{Universit\"at Duisburg-Essen, Essen, Germany\\
\emph{carlos.de-vera-piquero@uni-due.de}}\\
\end{center}

\begin{abstract}
Given a modular form $f$ of even weight larger than two and an imaginary quadratic field $K$ satisfying a relaxed Heegner hypothesis, we construct a collection of CM cycles on a Kuga--Sato variety over a suitable Shimura curve which gives rise to a system of Galois cohomology classes attached to $f$ enjoying the compatibility properties of an Euler system. Then we use Kolyvagin's method \cite{kolyvagin1990grothendieck}, as adapted by Nekov{\'a}{\v{r}} \cite{nekovar1992kolyvagin} to higher weight modular forms, to bound the size of the relevant Selmer group associated to $f$ and $K$ and prove the finiteness of the (primary part) of the Shafarevich--Tate group, provided that a suitable cohomology class does not vanish.
\end{abstract}

\setcounter{tocdepth}{1}

\tableofcontents

\section{Introduction}\label{intro}

Given a modular form $f$ of even weight, one strives to relate certain algebraic and analytic invariants associated with $f$. The classical expected relations correspond to conjectures formulated by Beilinson, Bloch and Kato, while their $p$-adic analogues were predicted by Perrin-Riou. Several results provide nowadays evidence towards these conjectures, and in most of them the theory of complex multiplication, giving rise to Heegner points or cycles, plays a prominent role.

The algebraic invariants alluded to above are usually related to bounds for the Selmer group associated with (the Galois representation attached to) $f$, while the analytic ones are concerned with the order of vanishing of the (complex or $p$-adic) $L$-series associated with $f$, or with its special values or its derivatives. Contributions to conjectures of this flavour frequently use appropriate special cycles as a bridge between the algebraic and analytic invariants. 

In this note, we extend Kolyvagin's method of Euler systems \cite{kolyvagin1990grothendieck} adapted by Nekov{\'a}{\v{r}} for modular forms of higher even weight \cite{nekovar1992kolyvagin} to the setting where the Heegner hypothesis is relaxed. We exploit Kuga--Sato varieties over Shimura curves in order to construct a Heegner system, that is, a collection of algebraic cycles satisfying certain local and global norm compatibility properties, from which one can extract arithmetic information about the Selmer group.

In order to fit our contribution into the above framework, let us first recall briefly some previous results that have a clear influence in the present work. The simplest scenario in which the above conjectures have been explored is of course the most down-to-earth setting of elliptic curves, or more generally of modular forms of weight 2. In this case, Kolyvagin \cite{kolyvagin1990grothendieck, gross1991} showed how to bound the Selmer group by exploiting the properties of a system of cohomology classes arising from Heegner points on the relevant modular curve (today commonly referred to as an Euler system of Kolyvagin type). Combined with the Gross--Zagier formula \cite{gross1986heegner}, relating the first derivative of the classical $L$-function associated with $f$ to the height of an appropriate Heegner point, and together with analytic non-vanishing results of Murty--Murty \cite{murty1991mean}, the Birch and Swinnerton-Dyer conjecture over $\QQ$ for elliptic curves of analytic rank at most 1 was established. On the $p$-adic side, an analogue of the Gross--Zagier formula was established by Perrin-Riou \cite{perrinriou1987points}. 

For modular forms of higher (even) weight, Kolyvagin's method was carefully extended by Nekov{\'a}{\v{r}} in \cite{nekovar1992kolyvagin}, by replacing the usual Selmer group of an elliptic curve with its cohomological higher weight analogue and the use of Heegner points on modular curves by the so-called Heegner cycles on suitable Kuga--Sato varieties, whose middle cohomology contain the Galois representations associated with higher weight modular forms. A Gross--Zagier formula, due to Zhang \cite{zhang1997heights}, holds also in this setting, and Nekov{\'a}{\v{r}} \cite{nekovavr1995thep} proved a $p$-adic avatar of this result. Combined with results of Bump, Friedburg and Hoffstein \cite{bump1990nonvanishing}, this provides further grounds for the conjectures for analytic rank less than or equal to 1. Still in the higher weight case, but in a different direction, Shnidman \cite{Shnidman2014padic} has recently developed classical and $p$-adic Gross--Zagier formulas for twists of modular forms by algebraic Hecke characters, while the first author \cite{Elias2014Selmer} has explored Kolyvagin's method to bound the size of the Selmer group also in this twisted situation. 

A key element in all the above works is the Heegner hypothesis that allows for the existence of Heegner points on the relevant modular curves (and hence, in the higher weight setting, of Heegner cycles on the relevant Kuga--Sato varieties). When this Heegner hypothesis fails, one can still use Shimura curves to provide a larger supply of modular parametrizations under a more relaxed assumption. In the case of elliptic curves, for instance, Heegner points arising from Shimura curve parametrizations give rise to algebraic points which could not be obtained by using modular curve parametrizations (see \cite[Chapter~4]{Darmon2004rational}, e.g.). 

In the Shimura curve setting, the above picture has been successfully adapted in the weight 2 case. Namely, Kolyvagin's method has been generalized to Hilbert modular forms (of parallel weight 2) over totally real fields by Nekov{\'a}{\v{r}} \cite{nekovar2007euler}, and X. Yuan, S.-W. Zhang and W. Zhang \cite{yuan2013gross} have proved a complete Gross--Zagier formula on quaternionic Shimura curves over totally real fields, building on previous work of S.-W. Zhang \cite{zhang2001heights}. On the $p$-adic side, it is worth mentioning that Disegni \cite{disegni2013padic} has recently proved a $p$-adic Gross--Zagier formula in this setting relating the central derivative of the $p$-adic Rankin--Selberg $L$-series associated with the modular form $f$ and the relevant CM extension to the $p$-adic height of a Heegner point on the abelian variety associated with $f$.

\vspace{0.2cm}

Next we describe the main result of this note. To do so, consider a newform $f_{\infty} \in S_{2r+2}^{\mathrm{new}}(\Gamma_0(N))$ of weight $2r+2\geq 4$ and level $\Gamma_0(N)$. Let $p$ be an odd prime not dividing $N\cdot(2r)!$, and let $\wp$ be a prime ideal dividing $p$ in the number field $F$ generated by the Fourier coefficients of $f_{\infty}$. The Galois representation $V_{\wp}(f_{\infty})$ attached to $f_{\infty}$ (a 2-dimensional $F_{\wp}$-vector space) might be realized as a factor in the middle \'etale cohomology of a (suitably compactified) Kuga--Sato variety over the modular curve $X_0(N)$ (see \cite{scholl1990motives}). Alternatively, we can also realize $V_{\wp}(f_{\infty})$ as a factor in the middle \'etale cohomology of Kuga--Sato varieties over certain Shimura curves, following an approach as in Besser \cite{BesserCycles} and Iovita--Spie{\ss} \cite{IovitaSpiess}.

More precisely, let $N=N^+N^-$ be any factorization of $N$ as a product of relatively prime integers $N^+$, $N^-$ such that $N^-$ is the square-free product of an even number of primes, and consider the Shimura curve $X$ attached to an Eichler order of level $N^+$ in an indefinite quaternion algebra of discriminant $N^-$. The Jacquet--Langlands correspondence associates to $f_{\infty}$ a Hecke eigenform $f$ on $X$, whose Galois representation $V_{\wp}(f)$, isomorphic to $V_{\wp}(f_{\infty})$, arises as a factor in the middle \'etale cohomology of the $r$-th Kuga--Sato variety $\mathcal A^r$ over the Shimura curve $X$ (see Section \ref{sec:repns} for details).

For a number field $K$, let $\CH^{r+1}(\mathcal A^r/K)$ be the $(r+1)$-th Chow group of $\mathcal A^r$ over $K$. The Abel--Jacobi map induces a Hecke- and Galois-equivariant map
\[
\Phi_{f,K}: \CH^{r+1}(\mathcal A^r/K)_0 \otimes F_{\wp} \, \, \longrightarrow \, \, \H^1(K,V_{\wp}(f)),
\]
where the subscript $0$ indicates the subgroup of cycle classes which are homologically trivial, and on the target we consider continuous Galois cohomology (cf. Section \ref{sec:AJmap}). In this note, we focus our attention on the above map when $K$ is an imaginary quadratic field satisfying the relaxed Heegner hypothesis (Heeg) spelled out in Section \ref{sec:CMcycles}. Namely, we require that one can choose the factorization $N=N^+N^-$ as above so that every prime dividing $N^+$ (resp. $N^-$) splits (resp. is inert) in $K$.

In this situation, complex multiplication points on the Shimura curve $X$ give rise to a system of cycles in $\mathcal A^r$ algebraic over ring class fields of $K$ (cf. Section \ref{sec:CMcycles}), leading to a system of (Kolyvagin) cohomology classes in $\H^1(K,V_{\wp}(f))$. The construction of such cycles resembles the construction in \cite{schoen1993complex}, the difference being that here we must construct them on QM abelian surfaces. The bottom layer of this system of algebraic cycles arises in the work of Iovita--Spie{\ss} \cite{IovitaSpiess}, who obtain a $p$-adic Gross--Zagier formula when $p$ divides $N$, and Besser \cite{BesserCycles}, who shows that the $r$-th Griffiths group of $\mathcal A^r$ has infinite rank. Besides, the image of the above map $\Phi_{f,K}$ is contained in the Selmer group $\Sel_{\wp}(f,K) \subseteq \H^1(K,V_{\wp}(f))$ (cf. Section \ref{sec:Selmer}), and the collection of algebraic cycles alluded to before gives us a cycle $y \in \CH^{r+1}(\mathcal A^r/K)_0$ whose image $y_0 = \Phi_{f,K}(y) \in \Sel_{\wp}(f,K)$ under $\Phi_{f,K}$ lies in the ($-\varepsilon$)-eigenspace under the action of complex conjugation, where $\varepsilon$ stands for the sign in the functional equation for the $L$-series associated with $f$. Further, it plays a central role in our main theorem:

\begin{theorem}\label{mainthm}
With the above notations, suppose $y_0$ is non-torsion. Then $\mathrm{Im}(\Phi_{f,K})$ has rank $1$ and $\Sh_{\wp}(f,K)$ is finite. More precisely, we have
\[
(\mathrm{Im}(\Phi_{f,K}))^{\varepsilon} = 0 \quad \text{and} \quad (\mathrm{Im}(\Phi_{f,K}))^{-\varepsilon} = F_{\wp} \cdot y_0.
\]
\end{theorem}

In the statement, $\Sh_{\wp}(f,K)$ denotes the $\wp$-primary part of the Shafarevich--Tate group, defined as the cokernel of the map $\Phi_{f,K}: \CH^{r+1}(\mathcal A^r/K)_0 \otimes F_{\wp} \to \Sel_{\wp}(f,K)$.

As we mentioned, our result fits in the framework of the conjectures by Beilinson, Bloch, Kato, and Perrin-Riou. Combined with forthcoming work of Disegni on a $p$-adic Gross--Zagier formula in this setting, we expect to shed some light on these conjectures for higher weight modular forms, when the classical Heegner hypothesis does not hold.

It is also worth mentioning that the work of Bertolini, Darmon and Prasana \cite{bertolini2013generalized}, relating special values of $p$-adic $L$-series associated to twists of modular forms to the image by the $p$-adic Abel--Jacobi map of certain algebraic cycles arising in underlying motives, was adapted by Masdeu and Brooks \cite{masdeu2012cycles, brooks2013generalized} to the setting where the Heegner hypothesis is removed. In Masdeu's work, the prime $p$ divides the level of the modular form and therefore one needs to deal with a bad reduction setting, whereas in Brooks' work, $p$ is a prime of good reduction and therefore the techniques are of a rather different nature. In this framework, it would be interesting to relate special values of $p$-adic $L$-series to the images by the $p$-adic Abel--Jacobi map of the cycles that we construct in this note.

\vspace{0.3cm}

\noindent {\bf Acknowledgements.} We are very grateful to Henri Darmon and Victor Rotger for many useful discussions on the topic of this paper. We also thank Erick Knight for his helpful clarifications and suggestions on the proof of Lemma \ref{lemma:etalecoh-torsionfree}.

\section{Shimura curves and QM abelian surfaces}\label{sec:Shimuracurves}

We describe in this section the Shimura curves that will play a central role  throughout this note, namely Shimura curves associated with Eichler orders in indefinite rational quaternion algebras. We recall the usual interpretation of such curves as moduli schemes for abelian surfaces with quaternionic multiplication (also referred to as {\em fake elliptic curves}), and then focus on special points on such moduli spaces, namely, abelian surfaces with quaternionic multiplication and complex multiplication.

\subsection{General definitions}\label{sec:curves-defns}

Fix a pair of relatively prime integers $N^+$, $N^-$, such that $N^-$ is the square-free product of an even number of primes, and set $N=N^+N^-$. Let $B$ be a rational quaternion algebra of reduced discriminant $N^-$ (hence, indefinite), and fix a maximal order $\OO_B$ in $B$ and an Eichler order $\mathcal R \subseteq \OO_B$ of level $N^+$. 

For every rational place $v$, we set $B_v := B\otimes_{\QQ}\QQ_v$, and at each finite place $\ell$ we shall also write $\OO_{B,\ell} := \OO_B\otimes_{\ZZ}\ZZ_{\ell}$, $\mathcal R_{\ell} := \mathcal R\otimes_{\ZZ}\ZZ_{\ell}$. We shall fix at the outset an isomorphism $\vartheta_{\infty}: B_{\infty} \to \M_2(\RR)$, which exists because $B$ is indefinite, and also an isomorphism $B_{\ell} \to \M_2(\QQ_{\ell})$ for each prime $\ell\nmid N^-$, identifying $\mathcal R_{\ell}$ with the standard Eichler order of level $\ell^{\mathrm{val}_{\ell}(N)}$. Write $\hat{\ZZ} := \prod \ZZ_{\ell}$ for the profinite completion of $\ZZ$, and for any $\ZZ$-algebra $R$ put $\hat{R} := R \otimes_{\ZZ} \hat{\ZZ}$. Thus, for example, $\hat{\QQ}$ stands for the ring of finite $\QQ$-adeles. We also let $\hat{B} := B\otimes_{\QQ}\hat{\QQ}$. 

Let $\mathcal H^{\pm} = \CC - \RR$ be the (disjoint) union of the upper and lower complex half planes, which might be identified with the set of $\RR$-algebra homomorphisms $\Hom(\CC,\M_2(\RR))$, and consider the space of double cosets
\begin{equation}\label{XR}
X_{\mathcal R} = \left( \hat{\mathcal R}^{\times} \backslash \hat{B}^{\times}\times \mathcal H^{\pm}\right)/B^{\times} = \left( \hat{\mathcal R}^{\times} \backslash \hat{B}^{\times}\times \Hom(\CC,\M_2(\RR))\right)/B^{\times}.
\end{equation}
Here, $\hat{\mathcal R}^{\times}$ acts naturally on the left on $\hat B^{\times}$ by left multiplication, and $B^{\times}$ acts on the right on both $\hat{B}^{\times}$ (diagonally) and on $\mathcal H^{\pm}$ by linear fractional transformations under our fixed isomorphism $\vartheta_{\infty}$. This latter action corresponds to the action on $\Hom(\CC,\M_2(\RR))$ by conjugation (again under $\vartheta_{\infty}$).

It follows from the work of Deligne and Shimura that $X_{\mathcal R}$ admits a model over $\QQ$, which further is the coarse moduli scheme classifying abelian surfaces with quaternionic multiplication by $\mathcal R$. Let us recall precisely these terms.

\begin{definition}\label{defn:QMA}
Let $S$ be a $\QQ$-scheme. An {\em abelian surface with quaternionic multiplication} ({\em QM}, for short) by $\mathcal R$ is a pair $(A,\iota)$ consisting of an abelian scheme $A/S$ of relative dimension $2$ endowed with an optimal embedding $\iota: \mathcal R \to \End_S(A)$, giving an action of $\mathcal R$ on $A$. 
\end{definition}

\begin{remark}\label{rmk:polarization}
In the above definition, if $s$ is a geometric point of $S$, then the QM abelian surface $A_s$ corresponding to $s$ is endowed with a unique {\em principal polarization} which is compatible with the QM structure (see \cite{Milne-modp}). Because of this reason, we will drop the polarization off in our discussion, although the reader should keep in mind the existence of a unique polarization compatible with the QM.
\end{remark}

Consider the moduli problem of classifying QM abelian surfaces, given by the moduli functor 
\begin{equation}\label{functor:ab}
\mathcal F: \text{Schemes}/\QQ \, \, \longrightarrow \text{Sets}
\end{equation}
sending a $\QQ$-scheme $S$ to the set $\mathcal F(S)$ of isomorphism classes of abelian surfaces with QM by $\mathcal R$ over $S$. Here, an isomorphism between two abelian surfaces with QM $(A,\iota)$ and $(A',\iota')$ is an isomorphism $\psi: A \to A'$ of the underlying abelian surfaces preserving the $\mathcal R$-action on both $A$ and $A'$, i.e. such that $\psi \circ \iota(\alpha) = \iota'(\alpha)\circ \psi$ for all $\alpha \in \mathcal R$.

\begin{theorem}[\cite{ShimuraClassFields}, \cite{DeligneShimura}]
$X_{\mathcal R}$ admits a model $X = X_{N^+,N^-}/\QQ$, which is the coarse moduli scheme associated to the moduli problem corresponding to the functor $\mathcal F$. Furthermore, the {\em Shimura curve} $X/\QQ$ is a smooth, projective and geometrically connected scheme over $\QQ$.
\end{theorem}

\begin{remark}
Alternatively, $X_{\mathcal R}$ is also the coarse moduli scheme classifying abelian surfaces with quaternionic multiplication by the maximal order $\mathcal O_B$ together with a level $N^+$-structure. 
\end{remark}

For our purposes, it is useful to introduce an auxiliary Shimura curve classifying QM abelian surfaces with suitable extra structure in order to make the moduli problem {\em fine}. 

\begin{definition}\label{defn:QMAextra}
Let $S$ be a $\QQ$-scheme and $M\geq 3$ be an integer prime to $N$. An abelian surface with QM by $\mathcal R$ and {\em full} level $M$-structure over $S$ is a triple $(A,\iota,\bar{\nu})$, where $(A,\iota)$ is a pair as in Definition \ref{defn:QMA} and $\bar{\nu}: (\mathcal R/M\mathcal R)_S \to A[M]$ is an $\mathcal R$-equivariant isomorphism from the constant group scheme $(\mathcal R/M\mathcal R)_S$ to the group scheme of $M$-division points of $A$.
\end{definition}

The corresponding moduli problem is now given by the moduli functor 
\begin{equation}\label{functor:abM}
\mathcal F_M: \text{Schemes}/\QQ \, \, \longrightarrow \text{Sets}
\end{equation}
sending a $\QQ$-scheme $S$ to the set $\mathcal F_M(S)$ of isomorphism classes of triples over $S$ as in Definition \ref{defn:QMAextra}. In this case, this moduli functor is represented by a {\em fine} moduli scheme over $\QQ$, which we will denote $X^M = X^M_{N^+,N^-}/\QQ$. It is also a smooth and projective curve over $\QQ$, although it is not geometrically connected. One can give an adelic description of $X^M$ in terms of double cosets as we did above for $X$. By forgetting the extra level structure at $M$, there is a natural Galois covering of Shimura curves
\[
X^M \, \, \longrightarrow \, \, X, \quad (A,\iota,\bar{\nu}) \, \, \longmapsto \, \, (A,\iota),
\]
whose Galois group is isomorphic to $G(M) := \GL_2(\ZZ/M\ZZ)/{\pm 1}$, using that 
\[
(\mathcal R/M\mathcal R)^{\times} \simeq (\OO_B/M\OO_B)^{\times} \simeq  \GL_2(\ZZ/M\ZZ).
\]

Since this second moduli problem is fine, there exists a universal family of abelian surfaces with QM by $\mathcal R$ and full level $M$-structure over $X^M$, corresponding to $1_{X^M} \in \mathrm{Hom}(X^M,X^M)$ under the bijection $\mathcal F_M(X^M) \leftrightarrow \mathrm{Hom}(X^M,X^M)$. We shall refer to this family as the {\em universal QM abelian surface over} $X^M$, and we will denote it by
\[
\pi: \mathcal A \, \, \longrightarrow \, \, X^M.
\]
Given a geometric point $x: \Spec \, L \to X^M$, the fibre $\mathcal A_x := \mathcal A \times_x \Spec \, L$ is an abelian surface with QM by $\mathcal R$ and full level $M$-structure defined over $L$, representing the isomorphism class corresponding to the moduli of $x$.

\begin{remark}\label{rmk:XC}
Over the complex numbers, $X^{an}(\CC)$ is identified with the compact Riemann surface $\Gamma \backslash \mathcal H$, as complex algebraic curves, where $\Gamma = \Gamma_{N^+,N^-}\subseteq \SL_2(\RR)$ is the image under $\vartheta_{\infty}$ of the group of units of reduced norm $1$ in the Eichler order $\mathcal R$. Indeed, upon identifying $\mathcal H$ with $\SL_2(\RR)/\SO_2(\RR)$ and noticing that $\hat{\mathcal R}^{\times}\backslash \hat B^{\times}/B^{\times}$ is trivial, one can easily define from \eqref{XR} an analytic isomorphism from $X^{an}(\CC)$ to $\Gamma \backslash \mathcal H$. Similarly, $X^{M,an}(\CC)$ can be identified with a finite union of compact Riemann surfaces of the form $\Gamma^M\backslash \mathcal H$.
\end{remark}

\begin{remark}\label{rmk:genus2}
We have reviewed above the usual moduli interpretation of Shimura curves in terms of QM abelian surfaces. However, the category of abelian surfaces being equivalent to the category of {\em stable} curves of genus $2$, one could also regard the Shimura curve $X$ as the coarse moduli space over $\QQ$ for stable curves of genus $2$ with QM by $\mathcal R$. Then one could consider the universal genus $2$ curve with QM over $X^M$, say $\mathcal C \to X^M$.
\end{remark}

\subsection{QM abelian surfaces with complex multiplication}\label{sec:QMCM}

Suppose $(A,\iota)$ is an abelian surface over $\CC$ with QM by $\mathcal R$, so that it defines a point $P=[A,\iota] \in X(\CC)$. Recall that $\iota: \mathcal R \hookrightarrow \End(A)$ is an optimal embedding of rings, giving an action of $\mathcal R$ on $A$ by endomorphisms. It is well-known that in this situation either 
\begin{itemize}
\item[(i)] $A$ is simple, and $\End^0(A) := \End(A)\otimes_{\ZZ}\QQ = B$, or
\item[(ii)] $A$ is not simple, and $\End^0(A) \simeq \mathrm M_2(K)$ for some imaginary quadratic field $K$ which embeds in $B$.
\end{itemize}

In the second case, $A$ is said to have CM by the imaginary quadratic field $K$. It is well-known that if $A$ has CM by $K$, i.e. $\End^0(A) \simeq \mathrm M_2(K)$, then $A$ is isogenous to the square of an elliptic curve with CM by $K$ (and conversely). However, we are interested in the category of QM abelian surfaces up to isomorphism rather than up to isogeny, thus this characterization is not sufficient for our goals.

In other terms, let $\End_{\mathcal R}(A) = \End(A,\iota) \subset \End(A)$ denote the subring of endomorphisms which commute with the QM action, i.e. 
\[
 \End_{\mathcal R}(A) = \End(A,\iota) := \{\gamma \in \End(A): \gamma \circ \iota(\alpha) = \iota(\alpha) \circ \gamma \text{ for all } \alpha \in \mathcal R\}.
\]
Then $\End(A,\iota)$ is either $\ZZ$ or an order in an imaginary quadratic field $K$. These two cases correspond, respectively, to (i) and (ii) above. If $K$ is an imaginary quadratic field (splitting $B$) and $\End(A,\iota)\simeq R_c$, where $R_c\subseteq K$ denotes the order of conductor $c\geq 1$ in $K$, then $(A,\iota)$ is said to have complex multiplication (CM) by $R_c$, and $P=[A,\iota]$ is said to be a CM (or Heegner) point by $R_c$. We write $\mathrm{CM}(X,R_c)$ for the set of all such points. 

There is a one-to-one correspondence between the set $\mathrm{CM}(X,R_c)$ and the set of ($\mathcal R^{\times}$-conjugacy classes of) optimal embeddings of $R_c$ into $\mathcal R$, given by associating to $P = [A,\iota] \in \mathrm{CM}(X,R_c)$ the embedding 
\[
 \varphi_P: R_c \simeq \End(A,\iota) \, \, \hookrightarrow \, \, \End_{\mathcal R}(\mathrm H^1(A,\ZZ)) \simeq \mathcal R,
\]
normalized as in \cite[Definition 1.3.1]{JordanPhD}.

Fix a CM point $P=[A,\iota]=[A_{\tau},\iota_{\tau}] \in \mathrm{CM}(X,R_c)$, and assume that $(c,N)=1$. Then $\mathcal R$ might be regarded via $\varphi_P$ as a locally free right $R_c$-module of rank $2$, hence $\mathcal R \simeq R_c \oplus e\mathfrak a$ for some $e\in B$ and some fractional $R_c$-ideal $\mathfrak a$. If $A_{\tau} = \CC^2/\Lambda_{\tau}$, with $\Lambda_{\tau} = \iota(\mathcal R)v$, $v=(\tau, 1)^t$, then we find that 
\begin{equation}\label{splittinglattice}
 \Lambda_{\tau} = \iota(\mathcal R)v = \iota(\varphi_P(R_c))v \oplus e\iota(\varphi_P (\mathfrak a))v.
\end{equation}
Further, $\iota(\varphi_P(K)) \subset \CC$ embeds diagonally in $\M_2(\CC)$, because $\End(A,\iota) = R_c$ and $\iota(\mathcal R)\otimes \RR = \M_2(\RR)$, hence
\[
 \Lambda_{\tau} = \iota(\varphi_P(R_c))v \oplus \iota(\varphi_P (\mathfrak a))ev
\]
and it follows that $A$ is isomorphic to a product $E \times E_{\mathfrak a}$ of elliptic curves with CM by $R_c$, where $E(\CC)=\CC/R_c$ and $E_{\mathfrak a}(\CC)=\CC/\mathfrak a$. The action of $\mathcal R$ on $E \times E_{\mathfrak a}$ induces the natural left action of $\mathcal R$ on $R_c \oplus e\mathfrak a$.

\begin{remark}\label{rmk:genus2CM}
In line with Remark \ref{rmk:genus2}, if $C$ is a stable genus $2$ curve with QM (meaning that its Jacobian variety $\mathrm{Jac}(C)$ has an action of $\mathcal R$ by endomorphisms), then $C$ is said to have CM if the subring of endomorphisms of $\mathrm{Jac}(C)$ which commute with the QM action form an order in an imaginary quadratic field. Then it is not hard to see that $C$ is isomorphic to the union of two elliptic curves meeting transversally at their identities, and $\mathrm{Jac}(C)$ is identified with their product.
\end{remark}

\subsection{Isogenies of QM abelian surfaces}\label{sec:QMisogenies}

As we already pointed out above, an isomorphism $(A,\iota) \to (A',\iota')$ between two QM abelian surfaces is an isomorphism of the underlying varieties which preserves the quaternionic action. More generally, the same notion applies for isogenies: if $(A,\iota)$ and $(A',\iota')$ are abelian surfaces with QM by $\mathcal R$, then an isogeny $\psi: A \to A'$ is an {\em isogeny of QM abelian surfaces}, or a {\em QM-isogeny} for short, if $\psi\circ \iota(\alpha) = \iota'(\alpha)\circ \psi$ for all $\alpha \in \mathcal R$. We write $\Hom_{\mathcal R}(A,A')$ for the ring of homomorphisms from $A$ to $A'$ which commute with the $\mathcal R$-action, so that non-zero elements in this ring correspond to QM-isogenies from $A$ to $A'$.

\begin{lemma}
 Let $(A,\iota)$ and $(A',\iota')$ be two abelian surfaces with QM by $\mathcal R$, and suppose that they are QM-isogenous. Then 
 \[
  \Hom_{\mathcal R}(A,A') \otimes_{\ZZ} \QQ \simeq \End(A,\iota)\otimes_{\ZZ} \QQ \simeq \End(A',\iota')\otimes_{\ZZ}\QQ.
 \]
\end{lemma}
\begin{proof}
Pick any QM-isogeny $\psi: A' \to A$. Thanks to the existence of (unique) principal polarizations on $A$ and $A'$, compatible with the QM structure (cf. Remark \ref{rmk:polarization}), one might regard the dual isogeny of $\psi$ as a QM-isogeny $\psi^{\vee}: A \to A'$, giving rise to an inverse isogeny $\psi^{-1}: A \to A'$ in $\Hom_{\mathcal R}(A,A')\otimes_{\ZZ}\QQ$. Now notice first that the rule $\varphi \mapsto \psi^{-1} \circ \varphi \circ \psi$ establishes an isomorphism 
\[
\End(A,\iota)\otimes_{\ZZ} \QQ \simeq \End(A',\iota')\otimes_{\ZZ} \QQ.
\]

Secondly, $\varphi \mapsto \psi \circ \varphi$ defines an injective morphism of $\ZZ$-modules $\Hom_{\mathcal R}(A,A') \hookrightarrow \End(A,\iota)$. But both of them are either of rank $1$ (if $A$ does not have CM, neither does $A'$) or of rank $2$ (if $A$ has CM by some imaginary quadratic field, so does $A'$). Thus it follows that this injective morphism induces an isomorphism between $\Hom_{\mathcal R}(A,A') \otimes_{\ZZ} \QQ$ and $\End(A,\iota)\otimes_{\ZZ} \QQ$.
\end{proof}

\begin{corollary}
 Suppose $(A,\iota)$ is a QM-abelian surface with CM by (some order in) $K$, and let $(A',\iota')$ be a second QM-abelian surface which is QM-isogenous to $(A,\iota)$. Then $(A',\iota')$ also has CM by (some order in) $K$, and in particular 
 \[
  \Hom_{\mathcal R}(A,A') \otimes_{\ZZ} \QQ \simeq K.
 \]
\end{corollary}
\begin{proof}
 If $(A,\iota)$ has CM by $K$, then $\End(A,\iota)$ is an order in $K$, and therefore $\End(A,\iota)\otimes_{\ZZ} \QQ \simeq K$. Then the statement follows directly from the previous lemma.
\end{proof}

\subsection{The N\'eron--Severi group of a QM abelian surface with CM}\label{sec:NSQM}

Let $P=[(A,\iota)]\in \mathrm{CM}(X,R_c) \subseteq X(K_c)$ be a CM point on $X$ by $R_c$, represented by a QM abelian surface $(A,\iota)$ with $\End(A,\iota)=R_c$, where $R_c$ denotes as before the order of conductor $c$ in the imaginary quadratic field $K$.

We write as above $A \simeq E \times E_{\mathfrak a}$ for some fractional $R_c$-ideal $\mathfrak a$. The N\'eron--Severi group $\NS(A)$ of the abelian surface $A$ is then identified with 
\[
 \NS(E\times E_{\mathfrak a}) \simeq \ZZ(E\times 0)\oplus \ZZ(0\times E_{\mathfrak a})\oplus \Hom(E,E_{\mathfrak a}) \simeq \ZZ(E\times 0)\oplus \ZZ(0\times E_{\mathfrak a})\oplus \mathfrak a,
\]
where in the first isomorphism an element $\gamma \in \Hom(E,E_{\mathfrak a})$ corresponds to the class of the divisor 
\[
 Z_{\gamma} := \Gamma_{\gamma} - (E\times 0) - \deg(\gamma)(0\times E_{\mathfrak a}) \subseteq E\times E_{\mathfrak a} \simeq A,
\]
with $\Gamma_{\gamma}$ standing for the graph of $\gamma$, and in the second isomorphism we use that $\Hom(E,E_{\mathfrak a}) \simeq \mathfrak a$.

Complex conjugation acts through the non-trivial element $\sigma \in \Gal(K/\QQ)$ on $\mathfrak a$, and then defines a decomposition $\mathfrak a = \mathfrak{a}_{+} \oplus \mathfrak{a}_{-}$, where $\mathfrak{a}_+$ (resp. $\mathfrak{a}_-$) is the $\ZZ$-submodule of $\mathfrak{a}$ on which $\sigma$ acts as multiplication by $+1$ (resp. -1). Then 
\[
 \mathfrak a_{+} = \ZZ a_{+} \quad \text{and} \quad \mathfrak a_{-} = \ZZ a_{-}
\]
for some elements $a_+ \in \QQ \cap \mathfrak a$ and $a_- \in \mathfrak a \subseteq K$, which we might regard either as elements in $K$ or as isogenies from $E$ to $E_{\mathfrak a}$. Notice that $a_-$ is purely imaginary. Rescaling the element $e \in B$ appearing in the decomposition of \eqref{splittinglattice} by a suitable non-zero scalar in $\QQ$ if necessary, we assume that $(a_{-})^2 = -D_c$, where $R_c = \ZZ[\sqrt{-D_c}]$. Therefore, we might rewrite $\mathfrak a$ under this convention as 
\[
\mathfrak a = \ZZ a \oplus \ZZ\sqrt{-D_c}, \quad \text{for some } a \in \QQ^{\times}.
\]
This normalization depends on the choice of a square-root $\sqrt{-D_c}$ of $-D_c$, and therefore is uniquely defined only up to sign. Observe also that since $\mathfrak a$ is a fractional $R_c$-ideal, we have $R_c\mathfrak a \subseteq \mathfrak a$, which implies in particular that in fact $a \in \ZZ$ and $a\mid D_c$, hence $a^{-1}D_c \in \ZZ$.

The N\'eron--Severi group of $A$ is then the (free) $\ZZ$-module of rank $4$ generated by (the classes of) the cycles $E\times 0$, $0\times E_{\mathfrak a}$, $Z_{a}$ and $Z_{\sqrt{-D_c}}$. Furthermore, the cycle $Z_{\sqrt{-D_c}}$ is orthogonal to the rank $3$ submodule $\langle E\times 0, 0\times E_{\mathfrak a}, Z_{a} \rangle$.

\section{Modular forms and $p$-adic Galois representations}

The main goal of this section is to explain how the $p$-adic Galois representation $V(f_{\infty})$ associated with a newform $f_{\infty} \in S_k(\Gamma_0(N))$ can be realized in the middle \'etale cohomology of a suitable Kuga--Sato variety over a Shimura curve, by following the approach of Besser \cite{BesserCycles} and Iovita--Spie{\ss} \cite{IovitaSpiess}. To do so, we first need to recall the Jacquet--Langlands correspondence and to introduce the Kuga--Sato varieties that will be involved.

\subsection{Modular forms and Jacquet--Langlands correspondence}

Fix an integer $N\geq 1$, and any factorization $N=N^+N^-$ of $N$ such that $\gcd(N^+,N^-)=1$ and $N^- > 1$ is the square-free product of an even number of primes. Associated to each of these factorizations, we can consider the Shimura curve $X_{N^+,N^-}/\QQ$ as above corresponding to an Eichler order $\mathcal R_{N^+,N^-}$ of level $N^+$ in the indefinite quaternion algebra $B$ of discriminant $N^-$. In this section we briefly recall the Jacquet--Langlands correspondence between classical cuspidal forms of level $\Gamma_0(N)$ and modular forms on the Shimura curve $X_{N^+,N^-}$.

Let $k=2r+2\geq 2$ be an even integer. In order to define modular forms of weight $k$ with respect to $\mathcal R_{N^+,N^-}$, identify the Lie algebra of left invariant differential operators on $B_{\infty}^{\times} :=(B\otimes_{\QQ}\RR)^{\times} \simeq \GL_2(\RR)$ with $\M_2(\CC)$, and define the differential operators
\[
W_{\infty} = \frac{1}{2}\left( \begin{smallmatrix} 0 & -\sqrt{-1} \\ \sqrt{-1} & 0 \end{smallmatrix} \right), \quad 
X_{\infty} = \left( \begin{smallmatrix} 1 & \sqrt{-1} \\ \sqrt{-1} & -1 \end{smallmatrix} \right).
\]

A (holomorphic) modular form of weight $k$ with respect to $\mathcal R$ is then a function 
\[
f: (B\otimes_{\QQ} \mathbb A_{\QQ})^{\times} = \hat B^{\times} \times \GL_2(\RR) \, \longrightarrow \, \CC
\]
satisfying the following properties:
\begin{itemize}
\item[i)] for every $b \in (B\otimes_{\QQ} \mathbb A_{\QQ})^{\times}$, the function $\GL_2(\RR) \to \CC$ defined by the rule $x \mapsto f(xb)$ is of $C^{\infty}$-class and satisfies $W_{\infty}f = (k/2)f$, $\bar{X}_{\infty}f = 0$;
\item[ii)] for every $\gamma \in B^{\times}$ and every $u \in \hat{\mathcal R}^{\times} \times \RR^{>0}$, $f(ub\gamma) = f(b)$.
\end{itemize}
The ($\CC$-)vector space of all modular forms of weight $k$ with respect to $\mathcal R_{N^+,N^-}$ will be denoted $S_k(X_{N^+,N^-})$.

Alternatively, by considering the congruence subgroup $\Gamma_{N^+,N^-} \subseteq \SL_2(\RR)$ and the identification of $X_{N^+.N^-}^{an}(\CC)$ with the (compact) Riemann surface $\Gamma_{N^+,N^-}\backslash \mathcal H$ as in Remark \ref{rmk:XC}, a modular form of weight $k$ with respect to $\mathcal R_{N^+,N^-}$ is the same as a holomorphic function $f: \mathcal H \, \longrightarrow \, \CC$ such that 
\[
f(\gamma \tau) = (c\tau+d)^kf(\tau) \quad \text{for all } \gamma = \left( \begin{smallmatrix} \ast & \ast \\ c & d \end{smallmatrix}\right) \in \Gamma_{N^+,N^-}.
\]

Under our assumption that $N^- > 1$, observe that no growth condition needs to be imposed at the cusps, since the Riemann surface $\Gamma_{N^+,N^-}\backslash \mathcal H$ is already compact.

The Shimura curve $X_{N^+,N^-}$ comes equipped with a ring of Hecke correspondences, which can be easily introduced by using the adelic description of $X_{N^+,N^-}$ given above (cf. \cite[Section 1.5]{BD_MumfordTate}). Such correspondences give rise to Hecke operators on the spaces of modular forms $S_k(X_{N^+,N^-})$. Indeed, the discrete double coset space $\hat{\mathcal R}_{N^+,N^-}^{\times}\backslash \hat B^{\times} / \hat{\QQ}^{\times}$ might be written as a product of local double coset spaces
\begin{equation}\label{localdoublecoset}
(\mathcal R_{N^+,N^-} \otimes \ZZ_{\ell})^{\times} \backslash (B\otimes \QQ_{\ell})^{\times} /\QQ_{\ell}^{\times},
\end{equation}
and this decomposition allows to define local correspondences, at each rational prime, that extend to global correspondences on $X_{N^+,N^-}$.

For each prime $\ell \nmid N$, the space \eqref{localdoublecoset} is identified with $\PGL_2(\ZZ_{\ell})\backslash \PGL_2(\QQ_{\ell})$, which in turn corresponds to the Bruhat--Tits tree of $\PGL_2(\QQ_{\ell})$. Thus there is a natural degree $\ell+1$ correspondence, sending each vertex $g\in\PGL_2(\ZZ_{\ell})\backslash \PGL_2(\QQ_{\ell})$ to the formal sum of its $\ell+1$ neighbours, denoted by $T_{\ell}$. This extends to a correspondence on $X_{N^+,N^-}$ of degree $\ell+1$, still denoted $T_{\ell}$. At each prime $q\mid N^+$, \eqref{localdoublecoset} is identified instead with the set of chains of edges of length $a$ in the Bruhat--Tits tree of $\PGL_2(\QQ_q)$, if $q^a\mid\mid N^+$. There is a natural involution on this set, corresponding to reversing the orientation of the edges in the Bruhat--Tits tree, and this extends again to an involution on $X_{N^+,N^-}$, that will be denoted by $W_q$. One can also define a correspondence $U_q$ for such primes; if $q$ divides exactly $N^+$, for instance, then $U_q$ is the degree $q$ correspondence defined by sending an edge $e$ in the Bruhat--Tits tree to the formal sum of the $q$ edges $e'\neq e$ having the same source as $e$. Finally, at primes $q\mid N^-$, the local space \eqref{localdoublecoset} consists only of two elements and the only involution defined on such set, which extends to an involution on the Shimura curve $X_{N^+,N^-}$, will be denoted also by $W_q$.

The Hecke operators $T_{\ell}$ for primes $\ell\nmid N$ are referred to as the {\em good} Hecke operators, whereas the operators $U_q$ are commonly named {\em bad} Hecke operators. The involutions $W_q$ are the so-called Atkin--Lehner involutions, and form a group of automorphisms $\mathcal W \simeq (\ZZ/2\ZZ)^{\omega(N)}$ of $X_{N^+,N^-}$, where $\omega(N)$ is the number of prime factors of $N$. Both the good and the bad operators, as well as the Atkin--Lehner involutions, act also as endomorphisms on the spaces $S_k(X_{N^+,N^-})$ of weight $k$ modular forms. We denote by $\mathbb T_{N^+,N^-}$ the $\ZZ$-algebra generated by the good Hecke operators $T_{\ell}$ together with the Atkin--Lehner involutions $W_q$.

The Jacquet--Langlands correspondence establishes a Hecke-equivariant bijection between automorphic forms on $\GL_2$ and its twisted forms. In our setting, this boils down to a correspondence between classical modular forms and quaternionic modular forms as stated below. 

\begin{proposition}[Jacquet--Langlands]\label{prop:JL}
For each factorization $N=N^+N^-$ as above, there is a $\mathbb T_{N^+,N^-}$-equivariant isomorphism (uniquely determined up to scaling)
\[
\mathrm{JL}: S_k(\Gamma_0(N^{-}N^{+}))^{(N^{-})-new} \quad \stackrel{\simeq}{\longrightarrow} \quad S_k(X_{N^+,N^-}).
\]
In particular, to each eigenform $f \in S_k(\Gamma_0(N^{-}N^{+}))^{(N^{-})-new}$ there corresponds a unique quaternionic form $f^B = \mathrm{JL}(f) \in S_k(X_{N^+,N^-})$ having the same Hecke eigenvalues as $f$ for the good Hecke operators $T_{\ell}$ ($\ell\nmid N$) and the Atkin--Lehner involutions $W_q$.
\end{proposition}

Let $F$ be a subfield of $\CC$, and write $S_k(\Gamma_0(N^{-}N^{+}),F) \subseteq S_k(\Gamma_0(N^{-}N^{+}))$ for the subspace of modular forms whose Fourier coefficients generate a subfield of $F$. The isomorphism $\mathrm{JL}$ above is compatible with Galois action, and hence  $S_k(X_{N^+,N^-},F) := \mathrm{JL}(S_k(\Gamma_0(N^{-}N^{+}),F)^{(N^{-})-new})$ must be regarded as the subspace of weight $k$ modular forms on $X_{N^+,N^-}$ which are defined over $F$, although such modular forms have no Fourier expansion. The Jacquet--Langlands correspondence then restricts to an isomorphism 
\[
\mathrm{JL}: S_k(\Gamma_0(N^{-}N^{+}),F)^{(N^{-})-new} \quad \stackrel{\simeq}{\longrightarrow} \quad S_k(X_{N^+,N^-},F).
\]

We can also reformulate the above Jacquet--Langlands correspondence in the following way. Suppose $f_{\infty} \in S_k(\Gamma_0(N))^{\mathrm{new}}$ is a normalized newform, which is an eigenform for the Hecke operators $T_{\ell}$, for $\ell\nmid N$, and the Atkin--Lehner involutions $W_q$, for $q\mid N$ (here, $W_q$ stands for the $W$-operator corresponding to the $q$-primary part $Q = q^a$ of $N$ as in \cite{AtkinLi}, which induces an involution on $S_k(\Gamma_0(N))$). Then we have $T_{\ell}f_{\infty} = a_{\ell}f_{\infty}$ for every $\ell \nmid N$ and $W_qf_{\infty} = \varepsilon_{q,f_{\infty}} f_{\infty}$ for every $q\mid N$, where $a_{\ell}= a_{\ell}(f_{\infty})$ stands for the $\ell$-th Fourier coefficient of $f_{\infty}$ and $\varepsilon_{q,f_{\infty}}=\pm 1$ is the eigenvalue of the Atkin--Lehner involution acting on $f_{\infty}$. If $\mathrm{val}_q(N) \geq 2$, then we have $a_q = 0$, whereas if $\mathrm{val}_q(N) = 1$ (in particular, for primes $q\mid N^-$), then the eigenvalue $\varepsilon_{q,f_{\infty}}$ is related to the $q$-th Fourier coefficient by the identity $a_q = -\varepsilon_{q,f_{\infty}}q^r$ (recall that according to our notations $k=2r+2$).

Let $F$ be the number field generated by the Fourier coefficients $a_n=a_n(f_{\infty})$ of $f_{\infty}$, which lie acually in its ring of integers $\mathcal O_F$. Then the Jacquet--Langlands correspondence asserts that for each factorization $N = N^+N^-$ as above there exists a (unique) algebra homomorphism 
\[
\phi_{N^+,N^-}: \mathbb T_{N^+,N^-} \to \mathcal O_F
\]
such that $\phi_{N^+,N^-}(T_{\ell}) = a_{\ell}$ if $\ell\nmid N$ and $\phi_{N^+,N^-}(W_q) = \varepsilon_{q,f_{\infty}}$ if $q \mid N$.

In particular, at primes $q\mid N^-$ (and also at primes $q$ dividing exactly $N^+$) we recover the $q$-th Fourier coefficient of $f_{\infty}$ as $a_q = -\phi_{N^+,N^-}(W_q)q^r$. Also, notice that the eigenvalue $w_{N,1}:= \phi_{N,1}(W_N)$ of the {\em Fricke involution} $W_N$ acting on $f_{\infty}$ coincides with $w_{N^+,N^-}:=\phi_{N^+,N^-}(W_N)$ (by a slight abuse of notation we write $W_N$ for both the Fricke involution acting on $f_{\infty}$ and the one acting on $f$, namely the product of all the involutions $W_q$ for primes $q\mid N$).

\subsection{Kuga--Sato varieties}

We fix from now on a factorization $N=N^+N^-$ as in the previous section, and denote by $\mathcal R=\mathcal R_{N^+,N^-}$ an Eichler order of level $N^+$ in the indefinite rational quaternion algebra of discriminant $N^-$ and by $X = X_{N^+,N^-}$ the corresponding Shimura curve. Let $\pi: \mathcal A \to X^M$ be the universal abelian surface over the Shimura curve $X^M$ as in Section \ref{sec:curves-defns}, where $M\geq 3$ is an auxiliary integer prime to $N$. Thus $\mathcal A/X^M$ is a relative scheme of relative dimension $2$ (and absolute dimension $3$), and for each geometric point $x: \Spec \, L \to X^M$ the fibre $\mathcal A_x := \mathcal A \times_x \Spec \, L$ is an abelian surface with QM by $\mathcal R$ and full level $M$-structure defined over $L$, representing the isomorphism class corresponding to the point $x$.

Let $r\geq 1$ be an integer, and $\mathcal A^r = \mathcal A \times_{X^M} \cdots \times_{X^M} \mathcal A$ be the $r$-th fibered product of $\mathcal A$ over $X^M$. We shall refer to $\mathcal A^r$ as the {\em $r$-th Kuga--Sato variety} over $X^M$. It has relative dimension $2r$ over $X^M$, and absolute dimension $2r+1$. A generic point in $\mathcal A^r$ might be represented as a tuple $(x;P_1,\dots,P_r)$, where $x$ is a point in $X$ and the $P_i$ are points in the fibre $\mathcal A_x$.

Suppose $\ell$ is a prime not dividing $MN$. Then we define the action of the Hecke operator $T_{\ell}$ on the Kuga--Sato variety $\mathcal A^r$ as follows. Let $X^{M, \ell}$ be the Shimura curve classifying triples $(A, \iota, C[\ell])$, where $(A,\iota)$ is a QM abelian surface parametrized by $X^M$, further endowed with a subgroup $C[\ell]$ of $A[\ell]$ which is stable under the action of $\mathcal R$ (via $\iota$) and cyclic as $\mathcal R$-module ($A[\ell]$ has $\ell+1$ such $\mathcal R$-submodules, all of them of order $\ell^2$). Notice that there is a natural forgetful morphism of Shimura curves $X^{M,\ell} \to X^M$. The fibre product $\mathcal A_{\ell} := \mathcal A \times_{X^M} X^{M,\ell}$ is then the universal abelian surface over $X^{M,\ell}$, equipped with a subgroup scheme $\mathcal C[\ell]$ of order $\ell^2$, which is also a module for the induced action of $\mathcal R$. Let $\mathcal Q$ denote the quotient of $\mathcal A_{\ell}$ by the subgroup scheme $\mathcal C[\ell]$, with level structure induced from $\mathcal A_{\ell}$. Write also $\mathcal A_{\ell}^r$ and $\mathcal Q^r$ for the respective $r$-th fibered products over $X^{M,\ell}$. Then the first and third squares in the following diagram are cartesian:

\[
\xymatrix{
\mathcal A^r \ar[d] & \ar[l]_{\phi_1} \ar[r]^{\phi_2} \ar[d]  \mathcal A^{r,\ell} & \ar[d] \ar[r]^{\phi_3} \mathcal Q^r  & \mathcal A^r\ar[d]\\
X^M & X^{M,\ell} \ar@{=}[r] \ar[l] & X^{M,\ell}  \ar[r] & X^M}
\]

By using this diagram, the Hecke operator $T_{\ell}$ acting on $\mathcal A^r$ can be defined then as the correspondence $T_{\ell} = \phi_{1}^{*} \circ \phi_{2*} \circ \phi_{3*}$. Such a correspondence induces an endomorphism, which we still denote $T_{\ell}$, on \'etale cohomology groups $\mathrm H^*_{\et}(\mathcal A^r \times \bar{\QQ}, -)$.

\subsection{$p$-adic Galois representations attached to modular forms}\label{sec:repns}

Let $f_{\infty} \in S_k(\Gamma_0(N),\bar{\QQ})$ be a (normalized) newform of even weight $k=2r+2$ ($r\geq 1$) and level $N=N^+N^-$ as before. Let also $f = \mathrm{JL}(f_{\infty}) \in S_k(X,\bar{\QQ})$ be the corresponding newform on $X$ under the Hecke-equivariant isomorphism from Proposition \ref{prop:JL}. Let $F =  \QQ(\{a_n\})$ be the number field generated by the Fourier coefficients $a_n = a_n(f_{\infty})$ of $f_{\infty}$, which actually lie in its ring of integers $\mathcal O_F$.

Let $p$ be a rational prime not dividing $N$. Associated with $f_{\infty}$, there is a free $\mathcal O_F \otimes_{\ZZ} \ZZ_p$-module $V_{\infty} = V(f_{\infty})$ of rank $2$ equipped with a continuous $\mathcal O_F$-linear action of $G_{\QQ} = \Gal(\bar{\QQ}/\QQ)$. Indeed, the Galois representation $V_{\infty}$ arises as a factor of the middle \'etale cohomology of the Kuga--Sato variety obtained as the $2r$-fold fibre product of (a suitable smooth compactification of) the universal elliptic curve $\mathcal E$ (with full level $N$-structure) over the modular curve $X(N)$ (cf. \cite{scholl1990motives}). A similar construction is available for $f = \mathrm{JL}(f_{\infty})$, by considering the universal abelian surface $\pi: \mathcal A \to X^M$ as a replacement for $\mathcal E \to X(N)$. We sketch below this construction, which is based on previous work of Besser \cite{BesserCycles} and follows the approach taken later in Iovita--Spie{\ss} \cite{IovitaSpiess}.

The action of $\mathcal R$ on $\mathcal A$ induces an action of $B^{\times}$ on $R^q\pi_*\QQ_p$, for $q\geq 1$. Then one defines a $p$-adic sheaf 
\[
 \mathbb L_2 := \bigcap_{b\in B} \ker(b-\text{n}(b): R^2\pi_*\QQ_p \to R^2\pi_*\QQ_p),
\]
which is a $3$-dimensional local system on $X^M$. Then the non-degenerate pairing 
\[
 (\, , \, ): \mathbb L_2\otimes \mathbb L_2 \hookrightarrow R^2\pi_*\QQ_p\otimes R^2\pi_*\QQ_p \stackrel{\cup}{\to} R^4\pi_*\QQ_p \stackrel{tr}{\to} \QQ_p(-2)
\]
induces a Laplacian operator 
\[
 \Delta_r: \Sym^r\mathbb L_2 \, \longrightarrow \, (\Sym^{r-2}\mathbb L_2)(-2),
\]
and $\mathbb L_{2r} \subseteq \Sym^r\mathbb L_2$ is defined as the kernel of $\Delta_r$.

The $p$-adic $G_{\QQ}$-representation attached to the space $S_k(X,\bar{\QQ})$ of weight $k$ modular forms on $X$ is then by definition (cf. \cite[Definition 5.6]{IovitaSpiess})
\begin{equation}\label{Repn}
 \mathrm H^1_{\et}(X^M \times \bar{\QQ}, \mathbb L_{2r})^{G(M)} = p_{G(M)} \mathrm H^1_{\et}(X^M \times \bar{\QQ}, \mathbb L_{2r}),
\end{equation}
where we take $G(M)$-invariants by applying the projector 
\[
p_{G(M)} := \frac{1}{|G(M)|} \sum_{g \in G(M)} g \in \QQ[G(M)],
\]
regarded as a correspondence in $\Corr_X(\mathcal A^r,\mathcal A^r)$ (notice that there is an action of $G(M)$ on $\mathbb L_{2r}$, compatible with the action on $X^M$). This makes $\mathrm H^1_{\et}(X^M \times \bar{\QQ}, \mathbb L_{2r})^{G(M)}$ independent of the choice of $M$, and one further has
\[
\mathrm H^1_{\et}(X^M \times \bar{\QQ}, \mathbb L_{2r})^{G(M)} \simeq p_{G(M)} \varepsilon_{2r} \mathrm H_{\et}^{2r+1}(\mathcal A^r \times \bar{\QQ}, \QQ_p) = p_{G(M)} \varepsilon_{2r} \mathrm H_{\et}^{*}(\mathcal A^r \times \bar{\QQ}, \QQ_p),
\]
where $\varepsilon_{2r}$ is a suitable projector in $\Corr_{X^M}(\mathcal A^r,\mathcal A^r)$ encoding the construction of $\mathbb L_{2r}$. Also, we may remark that the Galois representation in \eqref{Repn} arises as the $p$-adic \'etale realization of a Chow motive $\mathcal M_{2r}$ (see \cite[Appendix 10.1]{IovitaSpiess}).

The construction sketched above can be adapted to work with $\ZZ_p$-coefficients, and one can even consider cohomology with finite coefficients, provided that we assume that $p$ does not divide $N\cdot(2r)!$ and we choose the auxiliary level $M$ so that $p$ does not divide $|G(M)|$ (hence the projector $p_{G(M)}$ is well-defined in $\ZZ_p[G(M)]$ and $\ZZ/p^m\ZZ[G(M)]$). Write $\mathcal L_{2r,m}$ for the sheaf constructed as $\mathbb L_{2r}$ but with ring of coefficients $\ZZ/p^m\ZZ$, and let $\mathcal L_{2r} := \varprojlim \mathcal L_{2r,m}$ be the corresponding $p$-adic sheaf.

\begin{lemma}\label{lemma:etalecoh-torsionfree}
 With the previous notations, $\H^1_{\et}(X^M \times \bar{\QQ}, \mathcal L_{2r})$ is torsion free and 
 \[
  \H^1_{\et}(X^M \times \bar{\QQ}, \mathcal L_{2r,m}) = \H^1_{\et}(X^M \times \bar{\QQ}, \mathcal L_{2r})/p^m.
 \]
\end{lemma}
\begin{proof}
First of all, the sheaf $\mathcal L_{2r} = (\mathcal L_{2r,m})$ of $\ZZ_p$-modules is flat,
and hence for every $m > 1$ the exact sequence $0 \longrightarrow \ZZ/p^{m-1}\ZZ \stackrel{p}{\longrightarrow} \ZZ/p^{m}\ZZ \longrightarrow \ZZ/p\ZZ \longrightarrow 0$ induces when tensoring with $\mathcal L_{2r,m}$ a short exact sequence 
\[
0 \longrightarrow \mathcal L_{2r,m-1} \longrightarrow \mathcal L_{2r,m} \longrightarrow \mathcal L_{2r,1} \longrightarrow 0.
\]
By passing to the induced long exact sequence in \'etale cohomology and using that $\H^i_{\et}(X^M \times \bar{\QQ}, \mathcal L_{2r,m}) = 0$ for $i=0,2$ (cf. \cite[Appendix 10]{IovitaSpiess}), one obtains a short exact sequence
\begin{equation}\label{eqn:exactseqH1}
0 \longrightarrow \H^1_{\et}(X^M\times\bar{\QQ},\mathcal L_{2r,m-1}) \longrightarrow \H^1_{\et}(X^M\times\bar{\QQ},\mathcal L_{2r,m}) \longrightarrow \H^1_{\et}(X^M\times\bar{\QQ},\mathcal L_{2r,1}) \longrightarrow 0.
\end{equation}

Secondly, let $X^{Mp^n}/\QQ$ be the Shimura curve with full level $Mp^n$-structure above $X$, for each integer $n\geq 1$. Then the natural forgetful morphism of Shimura curves $X^{Mp^n} \to X^M$ is a Galois cover with Galois group $G(p^n)=\GL_2(\ZZ/p^n\ZZ)/\pm 1$, and we write $\pi_n: \mathcal A_n \to X^{Mp^n}$ for the corresponding universal abelian surface. By a slight abuse of notation, we still write $\mathcal L_{2r,m}$ for the $p$-adic sheaf on $X^{Mp^n}$ constructed in the same way as for $X^M$ but starting with $R^2\pi_{n,*}\ZZ/p^m\ZZ$. Then for each pair of integers $m,n\geq 1$, $\H^1_{\et}(X^M\times\bar{\QQ},\mathcal L_{2r,m})$ is identified with $\H^1_{\et}(X^{Mp^n}\times\bar{\QQ},\mathcal L_{2r,m})^{G(p^n)}$, where the superscript $G(p^n)$ means that we take invariants by $G(p^n)$. Indeed, the Hochschild--Serre spectral sequence for Galois covers in \'etale cohomology (see \cite[III.2.20]{MilneEtaleCohomology}) gives an exact sequence
\[
\H^1(G(p^n),\H^0_{\et}(X^{Mp^n}\times\bar{\QQ},\mathcal L_{2r,m})) \longrightarrow \H^1_{\et}(X^M\times\bar{\QQ},\mathcal L_{2r,m}) \longrightarrow \H^1_{\et}(X^{Mp^n}\times\bar{\QQ},\mathcal L_{2r,m})^{G(p^n)} \longrightarrow
\]
\[
\longrightarrow \H^2(G(p^n),\H^0_{\et}(X^{Mp^n}\times\bar{\QQ},\mathcal L_{2r,m})),
\]
and using now that $\H^0_{\et}(X^{Mp^n} \times \bar{\QQ}, \mathcal L_{2r,m}) = 0$ we obtain an isomorphism of $\ZZ/p^m\ZZ$-modules $\H^1_{\et}(X^M\times\bar{\QQ},\mathcal L_{2r,m}) \simeq \H^1_{\et}(X^{Mp^n}\times\bar{\QQ},\mathcal L_{2r,m})^{G(p^n)}$.

Now the inclusion $\H^1_{\et}(X^M\times\bar{\QQ},\mathcal L_{2r,m-1}) \hookrightarrow \H^1_{\et}(X^M\times\bar{\QQ},\mathcal L_{2r,m})$ from \eqref{eqn:exactseqH1} gives an isomorphism of $\ZZ/p^{m-1}\ZZ$-modules $\H^1_{\et}(X^M\times\bar{\QQ},\mathcal L_{2r,m-1}) \simeq \H^1_{\et}(X^M\times\bar{\QQ},\mathcal L_{2r,m})[p^{m-1}]$. To see this, first notice that for any pair of integers $1 \leq i \leq n$, the stalk of $\mathcal L_{2r,i}$ at any geometric point $x'$ of $X^{Mp^n}$ is isomorphic to the stalk at the image of $x'$ on $X^M$. However, the first one admits a trivialization induced from the trivialization of $R^2\pi_{n,*}\ZZ/p^i\ZZ$ (using the level $p^n$-structure on $\mathcal A_{n,x'}$). In particular, for $n\geq m$ one deduces that the inclusion 
\[
\H^1_{\et}(X^{Mp^n}\times\bar{\QQ},\mathcal L_{2r,m-1}) \hookrightarrow \H^1_{\et}(X^{Mp^n}\times\bar{\QQ},\mathcal L_{2r,m})
\]
in the exact sequence analogous to \eqref{eqn:exactseqH1} for $X^{Mp^n}$ becomes the inclusion of the $p^{m-1}$-torsion submodule $\H^1_{\et}(X^{Mp^n}\times\bar{\QQ},\mathcal L_{2r,m})[p^{m-1}]$ in $\H^1_{\et}(X^{Mp^n}\times\bar{\QQ},\mathcal L_{2r,m})$. By taking $G(p^n)$-invariants, the claimed assertion follows. 

Next we use this last observation to conclude that for all $m\geq 1$, $\H^1_{\et}(X^M\times\bar{\QQ},\mathcal L_{2r,m})$ is a free $\ZZ/p^m\ZZ$-module of rank $t = \dim_{\mathbb F_p}(\H^1_{\et}(X^M\times\bar{\QQ},\mathcal L_{2r,1}))$ (in particular, notice that $t$ is independent on $m$). Indeed, the statement is true for $m=1$, since $\H^1_{\et}(X^M\times\bar{\QQ},\mathcal L_{2r,1})$ is a finite dimensional $\mathbb F_p$-vector space. By induction, if we assume the statement true for $m-1$ and set $\H = \H^1_{\et}(X^M\times\bar{\QQ},\mathcal L_{2r,m})$ then we look at the short exact sequence 
\[
0 \longrightarrow \H[p^{m-1}] \longrightarrow \H \longrightarrow \H /\H[p^{m-1}] \longrightarrow 0
\]
given by \eqref{eqn:exactseqH1}. By the inductive hypothesis, the first term $\H[p^{m-1}] \simeq \H^1_{\et}(X^M\times\bar{\QQ},\mathcal L_{2r,m-1})$ is a free $\ZZ/p^{m-1}\ZZ$-module of rank $t$, and the third term is identified with $\H^1_{\et}(X^M\times\bar{\QQ},\mathcal L_{2r,1})$. Hence $\H^1_{\et}(X^M\times\bar{\QQ},\mathcal L_{2r,m})$ is necessarily a free $\ZZ/p^m\ZZ$-module of rank $t$.

Finally, the Lemma follows now directly by applying \cite[Lemma V.1.11]{MilneEtaleCohomology}.
\end{proof}

Under the running assumptions that $p$ does not divide $N\cdot(2r)!$ and $M$ is chosen such that $p$ does not divide $|G(M)|$ either, set 
\[
 J := p_{G(M)} \left(\mathrm H^1_{\et}(X^M \times \bar{\QQ}, \mathcal L_{2r})(r+1)\right) = \mathrm H^1_{\et}(X^M \times \bar{\QQ}, \mathcal L_{2r})(r+1)^{G(M)}.
\]

Let $\mathbb T=\mathbb T_{N^+,N^-}$ denote the Hecke algebra generated by the good operators $T_{\ell}$ ($\ell \nmid N$) and the Atkin--Lehner involutions $W_q$ ($q\mid N$), which act on $J$ by endomorphisms. Write $\mathbb I_f \subseteq \mathbb T$ for the kernel of the ring homomorphism $\phi: \mathbb T \to \mathcal O_F$ associated with $f$, and let $V = V(f) := \{x \in J: \mathbb I_f x = 0\}$ be the $f$-isotypical component of $J$. Since $f$ is a newform, there is a $\mathbb T[G_{\QQ}]$-equivariant projection $J \to V$ from $J$ to the $f$-isotypical component.

As in \cite[Lemma 5.8]{IovitaSpiess}, $V = V(f)$ is isomorphic to $V_{\infty} = V(f_{\infty})$, both as free $\mathcal O_F \otimes \ZZ_p$-modules of rank $2$ and as $G_{\QQ}$-representations. In particular, \cite[Proposition 3.1]{nekovar1992kolyvagin} applies verbatim for $V = V(f)$. We restate it here for convenience of the reader and for later reference. 

\begin{proposition}\label{prop:V}
\begin{itemize}
\item[i)] $V$ is a free $\mathcal O_F \otimes \ZZ_p$-module of rank $2$, equipped with a continuous $\mathcal O_F$-linear action of $\Gal(\bar{\QQ}/\QQ)$.
\item[ii)] There is a $\Gal(\bar{\QQ}/\QQ)$-equivariant, skew-symmetric pairing $[,]: V \times V \, \, \longrightarrow \, \, \ZZ_p(1)$ such that $[\lambda x, y] = [x, \lambda y]$ for all $x, y \in V$ and $\lambda \in \mathcal O_F\otimes\ZZ_p$, which induces non-degenerate pairings
\[
[,]_s: V/p^sV \times V/p^sV \, \, \longrightarrow \, \, \mu_{p^s}, \quad s\geq 0.
\]
\item[iii)] For each prime $\ell \nmid Np$, the characteristic polynomial of the (arithmetic) Frobenius automorphism $\mathrm{Fr}(\ell)$ acting on $V\otimes \QQ$ is 
\[
\det(1-x\mathrm{Fr}(\ell)\mid V\otimes\QQ) = 1-\frac{a_{\ell}}{\ell^r}x + \ell x^2.
\]
\item[iv)] For each prime $q \mid N$, 
\[
\det(1-x\mathrm{Fr}(q)\mid (V\otimes\QQ)_I) = 1-\frac{a_q}{q^r}x,
\]
with $a_q = 0$ or $-\varepsilon_{q,f}q^r$, where $I=\Gal(\bar{\QQ}_{q}/\QQ_{q}^{\mathrm{ur}})$ is the inertia subgroup and $\varepsilon_{q,f} = \phi_{N^+,N^-}(W_q) \in \{\pm 1\}$ is the eigenvalue of the Atkin--Lehner involution $W_q$ acting on $f$.
\end{itemize}
\end{proposition}

The above Proposition tells us that $V$ is a higher weight analogue of the Tate module of an elliptic curve, in which the Weil pairing is now replaced by the skew-symmetric pairing $[,]$ from ii). On the other hand, the isomorphism between $V$ and $V_{\infty}$ (both as $\mathcal O_F\otimes \ZZ_p$-modules and as $G_{\mathbb Q}$-representations) implies that $V\otimes \QQ$ is the $p$-adic realization of the motive $M_f:=M_{f_{\infty}}$ over $\QQ$ with coefficients in $F$ associated with $f_{\infty}$ (cf. \cite[p. 103]{nekovar1992kolyvagin} and \cite{scholl1990motives}). Proposition \ref{prop:V} then asserts that $M_f^{\vee} = M_f(-1)$ and 
\begin{equation}\label{eqn:Lf}
L(M_f^{\vee},s) = L(f_{\infty},s+r) = \sum_{n=1}^{\infty} a_n n^{-s-r}.
\end{equation}
According to the functional equation satisfied by the $L$-series $L(f_{\infty},s)$ (see \cite[Theorem 3.66]{shimura1971introduction}), we see that defining $\Lambda(s) := N^{s/2}(2\pi)^{-s-r}\Gamma(s+r)L(M_f^{\vee},s)$, the following relation holds 
\[
\Lambda(s) = \varepsilon \Lambda(2-s),
\]
with $\varepsilon = (-1)^{r+1} w_{N,1} = (-1)^{r+1} w_{N^+,N^-}$, where $w_{N,1}=\phi_{N,1}(W_N)$ (resp. $w_{N^+,N^-}=\phi_{N^+,N^-}(W_N)$) is the eigenvalue of the Fricke involution $W_N$ acting on $f_{\infty}$ (resp. $f = \mathrm{JL}(f_{\infty})$). In the remaining of this note, we will also write $L(f,s)$ for the complex $L$-series in \eqref{eqn:Lf}.

\begin{remark}
Working with the alternative moduli interpretation for the Shimura curve $X$ proposed in Remark \ref{rmk:genus2}, we could have considered the {\em Kuga--Sato variety} over $X^M$ obtained as the $2r$-th fibered self-product $\mathcal C^{2r}$ of the universal genus 2 curve $\mathcal C$ over $X^M$. The cohomologies of $\mathcal C^{2r}$ and $\mathcal A^r$ are closely related, and one could realize $V=V(f)$ as a factor in the middle \'etale cohomology of $\mathcal C^{2r}$ similarly as we did by using $\mathcal A^{r}$. When dealing with CM (or Heegner) cycles as defined in the next section, we found no advantage in taking this alternative approach (cf. Remark \ref{rmk:genus2CM}), although we believe it could be useful when dealing with diagonal cycles in products of Kuga--Sato varieties over Shimura curves along the lines of \cite{DarmonRotger}.
\end{remark}

\section{Algebraic cycles and $p$-adic \'etale Abel--Jacobi map}

In this section we will introduce and explore special algebraic cycles on the Kuga--Sato variety $\mathcal A^r$, which will be referred to as {\em CM} (or {\em Heegner}) {\em cycles} as they lie above CM (or Heegner) points on the Shimura curve $X$. Their definition resembles the well-known construction of CM cycles on Kuga--Sato varieties above classical modular curves, since they are essentially obtained from the CM cycles in the QM-abelian surfaces parametrized by CM points on the Shimura curve $X$ (cf. Section \ref{sec:Shimuracurves}). The images of these CM cycles under suitable $p$-adic \'etale Abel--Jacobi maps will give rise to a system of Galois cohomology classes that will be the input for Kolyvagin's method to bound the Selmer group, so let us first recall how these Abel--Jacobi maps are defined.

\subsection{$p$-adic \'etale Abel--Jacobi map}\label{sec:AJmap}

Let $W$ be a smooth proper irreducible variety of dimension $d$ defined over a field $K$, and let $\CH^c(W/K)$ be the Chow group of rational equivalence classes of codimension $c$  (algebraic) cycles in $W$ defined over $K$. If $\overline W$ denotes the base change of $W$ to a fixed algebraic closure $\bar K$ of $K$, the $p$-adic \'etale cycle class map 
\[
\CH^c(W/K) \, \, \longrightarrow \, \, \H^{2c}_{\et}(W,\ZZ_p(c))
\]
described in \cite{jannsen1988continuous} gives rise to a map 
\[
\mathrm{cl}: \CH^c(W/K) \, \, \longrightarrow \, \, \H^{2c}_{\et}(\overline W,\ZZ_p(c)).
\]
Here, $\H_{\et}^*$ stands for continuous \'etale cohomology, and we use the usual convention for Tate twists: if $M$ is a $\ZZ_p[G_{\QQ}]$-module and $c$ is an integer, then $M(c) = M \otimes \chi_p^c$, where $\chi_p: G_{\QQ} \to \ZZ_p^{\times}$ is the $p$-adic cyclotomic character.

A cycle whose class in $\CH^c(W/K)$ is in the kernel of $\mathrm{cl}$ is called {\em null-homologous}, and we write $\CH^c(W/K)_0 := \ker(\mathrm{cl})$ for the group of all such (classes of) cycles. The Hochschild--Serre spectral sequence (see \cite{jannsen1988continuous}) relating continuous \'etale cohomology with continuous Galois cohomology of $G_K = \Gal(\bar K/K)$, 
\[
\H^i(K, \H^j_{\et}(\overline W,\ZZ_p(c))) \Rightarrow \H^{i+j}_{\et}(W,\ZZ_p(c)),
\]
gives rise to the so-called ($c$-th) $p$-adic \'etale Abel--Jacobi map on $W$ over $K$
\[
\mathrm{AJ}_p: \CH^c(W/K)_0 \, \, \longrightarrow \, \, \H^1(K,\H_{\et}^{2c-1}(\overline W,\ZZ_p)(c)).
\]

For our purposes in this note we are interested in considering $W = \mathcal A^r$, the $r$-th Kuga--Sato variety over $X^M$, where $k=2r+2$ is the weight of our starting cuspidal newform $f_{\infty}$ of level $N$, and we still keep our fixed factorization $N=N^+N^-$. Then the $(r+1)$-th $p$-adic Abel-Jacobi map for $\mathcal A^r$ over a field $K$ of characteristic zero takes the form 
\begin{equation}\label{AJpKugaSato}
\mathrm{AJ}_p: \mathrm{CH}^{r+1}(\mathcal A^r/K)_0 \, \, \longrightarrow \, \, \H^1(K, \H^{2r+1}_{\et}(\mathcal A^r\times_K \bar K,\ZZ_p)(r+1)).
\end{equation}
Since the Abel--Jacobi map commutes with automorphisms of the underlying variety $\mathcal A^r$, by applying the projectors $\varepsilon_{2r}$ and $p_{G(M)}$ we see that $\mathrm{AJ}_p$ induces a map 
\[
\varepsilon_{2r}(\CH^{r+1}(\mathcal A^r/K)_0\otimes_{\ZZ}\ZZ_p) \, \, \longrightarrow \, \, \mathrm \H^1(K, J).
\]
But now notice that $\varepsilon_{2r}\mathrm H^{2r+2}_{\et}(\mathcal A^r\times_K \bar K,\ZZ_p)(r+1)) = 0$ (see \cite[Lemma 10.1]{IovitaSpiess}), or in other words, the target of the cycle class map on $\CH^{r+1}(\mathcal A^r/K)$ vanishes after applying $\varepsilon_{2r}$. Therefore, 
\[
\varepsilon_{2r}(\CH^{r+1}(\mathcal A^r/K)_0 \otimes_{\ZZ} \ZZ_p) = \varepsilon_{2r}(\CH^{r+1}(\mathcal A^r/K) \otimes_{\ZZ} \ZZ_p)
\]
and composing the last map with the projection $e_f: J \to V$ onto the $f$-isotypical component we finally obtain a map 
\[
\Phi_{f,K} : \CH^{r+1}(\mathcal A^r/K) \otimes_{\ZZ} \ZZ_p \, \, \stackrel{\varepsilon_{2r}}{\longrightarrow} \, \,  \varepsilon_{2r}(\CH^{r+1}(\mathcal A^r/K) \otimes_{\ZZ} \ZZ_p) \, \, \longrightarrow \, \, \H^1(K,V).
\]

Since $\mathrm{AJ}_p$ commutes with correspondences and with the Galois action, it follows that the map $\Phi_{f,K}$ is both $\mathbb T$-equivariant and (if $K/\QQ$ is Galois) $\Gal(K/\QQ)$-equivariant.

\subsection{Conjectures by Beilinson, Bloch, Kato and Perrin-Riou.}

The Beilinson conjectures predict a relation between the rank of the Chow group of homologically trivial algebraic cycles on a smooth projective variety $X$ of dimension $2i+1$ over a number field $K$ and the order of vanishing of the $L$-function attached to its \'{e}tale cohomology, expecting that 
\[
\ord_{s=i+1}\ L(\H_{\et}^{2i+1}(X), \ s) \stackrel{?}{=} \dim_{\overline{\QQ}} \ \CH^{i+1}(X)_0,
\]
(see \cite[Conjecture~5.9]{Beilinson1989height}). This conjecture can be refined by applying the projector $e$ corresponding to the motive associated to $X$ one is interested in to both sides of the equality to obtain the prediction that
\[
\ord_{s=i+1}L(e \ \H_{\et}^{2i+1}(X),s) \stackrel{?}{=} \dim_{\overline{\QQ}} e\ \CH^{i+1}(X)_0.
\]
Furthermore, letting $M$ denote $\H_{\et}^{2i+1}(X)$, and letting $\Sel(K, e \ M)$ be the submodule of $H^1(K, e \ M)$ of cohomology classes locally unramified everywhere outside a finite set $S$ of primes and satisfying suitable conditions for primes in $S$, Bloch and Kato predict that the Abel-Jacobi map 
\[
e \ \CH^{i+1}(X)_0 \otimes \QQ_p \longrightarrow \Sel(K, e \ M)
\]
is an isomorphism \cite{BK-Lfunctions}. Hence, one expects that 
\[
\ord_{s=i+1}L(e \ M,s) \stackrel{?}{=} \rank \ \Sel(K, e \ M).
\]

We outlined in the Introduction some of the contributions towards these conjectures in the setting where the motive $M$ is associated to an elliptic curve or a (twisted) higher even weight modular form and a CM field satisfying the appropriate Heegner hypothesis. In our setting, one strives to prove that
\[
\ord_{s=r+1} L(f \otimes K, s) \stackrel{?}{=} \rank \ \Sel^{(\infty)}_{\wp}(f,K),
\]
where $\Sel^{(\infty)}_{\wp}(f,K)$ is defined in \eqref{SelInfinity}.
For $r=0$, X. Yuan, S.-W. Zhang and W. Zhang \cite{yuan2013gross} showed that
\[
L'(f\otimes K,1) = \langle y_0, y_0 \rangle
\]
up to an explicit non-zero constant, for an appropriate height pairing $\langle \ , \ \rangle$, and Disegni \cite{disegni2013padic} obtained a $p$-adic avatar of this result. However, even if one disposed of a formula such as the one by X. Yuan, S.-W. Zhang and W. Zhang in higher weight, one would not be able to deduce the equality predicted by Beilinson, Bloch and Kato, because $\Phi_{f,K}$ is not known to be injective. On the other hand, one could tackle conjectures in the $p$-adic realm due to Perrin-Riou \cite{colmez1998fonctions} of the form
\[
L_p'(f\otimes K,\ell_K,1) \stackrel{?}{=} \langle \Phi_{f,K}(y_0), \Phi_{f,K}(y_0)\rangle
\]
where $\ell_K: \Gal(K_{\infty}/K) \longrightarrow \mathbb{Q}_p$ and $\langle \ , \ \rangle$ is a suitable $p$-adic height pairing. Indeed, a $p$-adic Gross--Zagier formula in higher even weight would combine with our result to validate an equality as above, along the lines of Perrin-Riou's conjecture for modular forms over Shimura curves, provided that the underlying $p$-adic height pairing is non-degenerate. This is the subject of forthcoming work of Daniel Disegni.

\subsection{CM cycles}\label{sec:CMcycles}

Now we construct the CM (or Heegner) cycles on the Kuga--Sato variety $\mathcal A^r$ alluded to above, sitting above CM points on $X=X_{N^+,N^-}$. To do so, we shall first recall the notion of CM points on the Shimura curve $X$, so we will fix from now on an imaginary quadratic field $K$ satisfying the Heegner hypothesis

\vspace{0.2cm}
\noindent
(Heeg) \hspace{0.3cm} {\em Every prime dividing $N^-$ (resp. $N^+$) is inert (resp. split) in $K$.} 
\vspace{0.2cm}

This assumption implies that $K$ embeds as a subfield in $B$. Consider then the natural map
\[
\hat B^{\times} \times \Hom(K,B) \, \, \longrightarrow \, \, X(\CC) = \left( \hat{\mathcal R}^{\times} \backslash \hat B^{\times} \times \Hom(\CC,\M_2(\RR))\right)/B^{\times},
\]
given by extending homomorphisms $K \to B$ to homomorphisms $\CC \to \M_2(\RR)$. A point $x\in X(\CC)$ is called a CM point (or a {\em Heegner} point) by $K$ if $x=[b,g]$ for some pair $(b,g) \in \hat B^{\times} \times \Hom(K,B)$. We denote by 
\[
\mathrm{CM}(X,K) := \{[b,g] \in X(\CC): (b,g) \in \hat B^{\times} \times \Hom(K,B)\}
\]
the set of all CM points on $X$ by $K$. Let $c\geq 1$ be an integer with $\gcd(N,c)=1$. A CM point $x=[b,g]\in \mathrm{CM}(X,K)$ is said to be of conductor $c\geq 1$ if the embedding $g: K \to B$ is optimal with respect to $R_c$ and $\mathcal R$, where $R_c$ denotes the quadratic order in $K$ of conductor $c$. This optimality condition means that 
\[
g(K) \cap b^{-1}\hat{\mathcal R}b = g(R_c),
\]
so that no quadratic order strictly containing $R_c$ as suborder embeds into $b^{-1}\hat{\mathcal R}b$ via $g$. Using that Eichler orders in indefinite rational quaternion algebras have class number one, it is easy to show that the set $\mathrm{CM}(X,K,c) \subseteq \mathrm{CM}(X,K)$ of Heegner points of conductor $c$ is in bijection with the set $\mathrm{Emb}(R_c,\mathcal R)$ of ($\mathcal R^{\times}$-conjugacy classes of) optimal embeddings of $R_c$ into $\mathcal R$. In particular, $|\mathrm{CM}(X,K,c)| = 2^th(R_c)$, where $t$ denotes the number of primes dividing $N$ and $h(R_c)$ is the class number of $R_c$.

Shimura's reciprocity law (cf. \cite[3.9]{DeligneShimura}, \cite[II.5.1]{MilneModels}, \cite[1.10]{MilneGoodredn}) asserts that $\mathrm{CM}(X,K)\subseteq X(K^{ab})$ and that the Galois action of $\Gal(K^{ab}/K)$ on Heegner points is described via the reciprocity map 
\[
\mathrm{rec}_K: \hat{K}^{\times} \longrightarrow \Gal(K^{ab}/K) 
\]
from class field theory. More precisely, for every $a\in \hat K^{\times}$ and every CM point $[b,g] \in \mathrm{CM}(X,K)$, 
one has 
\[
\mathrm{rec}_K(a)[b,g] = [\hat g(a)b,g],
\]
where $\hat g: \hat K^{\times} \to \hat B^{\times}$ denotes the embedding induced by $g$. In particular, recall that for each $c\geq 1$ the reciprocity map induces an isomorphism 
\[
\hat K^{\times}/\hat R_c^{\times}K^{\times} \simeq \Pic(R_c) \, \, \stackrel{\simeq}{\longrightarrow} \Gal(K_c/K),
\]
where $K_c$ denotes the ring class field of $K$ of conductor $c$, hence it follows that 
\[
\mathrm{CM}(X,K,c) \subseteq X(K_c).
\]

Moreover, one can check that $\Gal(K_c/K)$ acts freely on the set of Heegner points of conductor $c$, so that $\mathrm{CM}(X,K,c)$ has $2^t$ distinct $\Gal(K_c/K)$-orbits. Besides, the group $\mathcal W$ of Atkin--Lehner involutions acts on $\mathrm{CM}(X,K,c)$ as well, and for each prime $q\mid N$, the corresponding involution $W_q$ acts on CM points by switching their local {\em orientation} at $q$. The Galois action, on the contrary, preserves the orientations, so that the $2^t$ distinct $\Gal(K_c/K)$-orbits are in one-to-one correspondene with the $2^t$ possible orientations (cf. \cite[Chapter 7]{TianThesis}).

From a moduli point of view, a CM point $x \in \mathrm{CM}(X,K,c) \subseteq X(K_c)$ of conductor $c$ corresponds to the isomorphism class $[A,\iota]$ of an abelian surface with QM by $\mathcal R$ having CM by the order $R_c$; that is, such that $\End(A,\iota) \simeq R_c$. We use this interpretation of CM points to construct certain algebraic cycles on Kuga--Sato varieties above $X$, the desired CM cycles. To do so, we use the isomorphism $A \simeq E \times E_{\mathfrak a}$ from Section \ref{sec:QMCM} in order to define a cycle $Z_A \subseteq A$, which will eventually give rise to a cycle in $\mathcal A^r$. Recall that in the previous isomorphism, $E$ and $E_{\mathfrak a}$ are elliptic curves with CM by $R_c$, such that $E(\CC)=\CC/R_c$ and $E_{\mathfrak a}(\CC) = \CC/\mathfrak a$. Here $\mathfrak a$ is a fractional $R_c$-ideal.

We shall adopt the same conventions and normalizations as in Section \ref{sec:NSQM}. Hence by choosing a square root $\sqrt{-D_c} \in R_c$ we can decompose $\mathfrak a = \ZZ a \oplus \ZZ\sqrt{-D_c}$ as a direct sum of its $(+1)-$ and $(-1)-$eigencomponents for complex multiplication (or equivalently, for the $\Gal(K/\QQ)$-action), where $a \in \ZZ$ divides $D_c$. As explained in Section \ref{sec:NSQM}, the N\'eron--Severi group of $A$ is the free $\ZZ$-module of rank $4$ generated by the classes of the cycles 
\[
E \times 0, \quad 0 \times E_{\mathfrak a}, \quad Z_{a} \quad \text{and} \quad Z_{\sqrt{-D_c}}.
\]
We then define the CM cycle associated with $A$ to be the cycle
\[
Z_A := Z_{\sqrt{-D_c}}\subseteq A.
\]
Notice that under our conventions, the cycle $Z_A$ is well-defined up to sign, since its sign changes when $\sqrt{-D_c}$ is replaced by $-\sqrt{-D_c}$.

\begin{remark}
If $A$ is a QM abelian surface, then $\NS(A)_{\QQ}:=\NS(A)\otimes_{\ZZ}\QQ$ has a natural right $B^{\times}$-action. If further $A$ has CM as above, then one has a decomposition (as a right $B^{\times}$-module) $\NS(A)_{\QQ} \simeq \mathrm{ad}^0(B) \oplus \mathrm{Nrd}$, where $\mathrm{ad}^0(B)$ denotes the representation of $B^{\times}$ consisting of elements of trace zero on which $ B^{\times}$ acts on the right by conjugation and $\mathrm{Nrd}$ denotes the one-dimensional representation of $B^{\times}$ given by the reduced norm (cf. \cite{BesserCycles}, \cite[Lemma 8.1]{IovitaSpiess}). The one-dimensional subspace of $\NS(A)_{\QQ}$ corresponding to $\mathrm{Nrd}$ is usually refered to as its ``CM part'', and is generated by (the class of) $Z_A$.
\end{remark}

Next we study how the cycles $Z_A$ are related when $A$ varies in a QM-isogeny class of abelian surfaces with QM by $\mathcal R$ and CM by $K$. In other words, we want to relate the Heegner cycles constructed as above when we vary $x$ in $\mathrm{CM}(X,K)$.

First of all, we notice that the choice of the square root $\sqrt{-D_c}$ does not only fix the cycle $Z_A$, but also the CM cycles $Z_{A'}$ for every QM-abelian surface $A'$ which is QM-isogenous to $A$. Indeed, if $\psi: A \to A'$ is a QM-isogeny, then we fix the sign of $Z_{A'}$ by insisting that $\psi_* Z_A = t Z_{A'}$ for some $t > 0$. This condition does not depend on the isogeny $\psi$; to see this, by using that QM abelian surfaces admit a unique principal polarization compatible with the QM structure, one is reduced to show that for every $\phi \in \End_{\mathcal R}(A)$ there exists a constant $t>0$ such that $\phi_*Z_A = t Z_A$. But the latter holds because $\phi_*$ acts on $\NS(A)$ as multiplication by $\deg(\phi)$. In other words, one only requires $\sqrt{-D_{c'}}/\sqrt{-D_c}$ to be positive under the canonical identification $R_c \otimes \QQ \simeq R_{c'} \otimes \QQ$.

Now fix a CM point $x' = [A',\iota'] \in \mathrm{CM}(X,K,c')$ of conductor $c'\geq 1$, represented by some QM abelian surface $(A',\iota')$ with $\End(A',\iota')\simeq R_{c'}$. Similarly as for $A$, we now write $A' \simeq E' \times E_{\mathfrak b}$, where $E'(\CC) = \CC/R_{c'}$ and $E_{\mathfrak b}(\CC) = \CC/\mathfrak b$ for some fractional $R_{c'}$-ideal $\mathfrak b$. We write $\mathfrak b = \ZZ b \oplus \ZZ\sqrt{-D_{c'}}$ following the same conventions as those for $\mathfrak a$. Then the CM cycle associated with $x'$, or with $A'$, is the cycle $Z_{A'} := Z_{\sqrt{-D_{c'}}}$ corresponding to the isogeny
\[
\sqrt{-D_{c'}} \in \mathfrak b \simeq \Hom(E',E_{\mathfrak b}).
\]
Choosing the square root $\sqrt{-D_{c'}} \in R_{c'}$ so that $\sqrt{-D_{c'}}/\sqrt{-D_c}$ is positive, the precise relation between the cycles $Z_A$ and $Z_{A'}$ under a given QM-isogeny $A \to A'$ is the content of the next proposition.

\begin{proposition}\label{CMcycles-isogenies}
Let $(A,\iota)$ and $(A',\iota')$ be as above, and let $\psi: A \to A'$ be a QM-isogeny of abelian surfaces with QM by $\mathcal R$ and CM by $K$. Then:
\begin{itemize}
\item[a)] $\psi_* Z_A = (\deg(\psi))^{1/2}(bD_c/aD_{c'})^{1/2}Z_{A'}$.
\item[b)]$\psi^* Z_{A'} = (\deg(\psi))^{1/2}(aD_{c'}/bD_c)^{1/2}Z_A$.
\end{itemize}
\end{proposition}
\begin{proof}
First notice that the degree of the isogeny $\sqrt{-D_c}: E \to E_{\mathfrak a}$ equals the index of $\sqrt{-D_c}R_c$ in $\mathfrak a$ as lattices. This index is precisely $a^{-1}D_c$. Similarly, the degree of $\sqrt{-D_{c'}}: E' \to E_{\mathfrak b}$ is $b^{-1}D_{c'}$. We see therefore that the constant $bD_c/aD_{c'}$ in the statement is precisely $\deg(\sqrt{-D_c})/\deg(\sqrt{-D_{c'}})$. Similarly, the ratio $aD_{c'}/bD_c$ in (b) coincides with $\deg(\sqrt{-D_{c'}})/\deg(\sqrt{-D_c})$.

Now consider the CM cycle $Z_A = \Gamma_{\sqrt{-D_c}} - (E\times 0) - \deg(\sqrt{-D_c})(0\times E_{\mathfrak a})$. A straightforward computation shows that $\psi_*Z_A$ is orthogonal with respect to the intersection pairing to $\langle E'\times 0, 0 \times E_{\mathfrak b}, Z_{b}\rangle$, hence it must lie in the $\ZZ$-submodule of rank $1$ in $\NS(A')$ generated by the CM cycle $Z_{A'}$. As a consequence, $\psi_*Z_A = q Z_{A'}$ for some constant $q$. Such a constant can be then obtained by computing the self-intersection numbers $Z_{A'}\cdot Z_{A'}$ and $(\psi_*Z_A)\cdot(\psi_*Z_A)$ and using the identity $(\psi_*Z_A)\cdot(\psi_*Z_A) = q^2 Z_{A'}\cdot Z_{A'}$. We have 
\[
Z_{A'}\cdot Z_{A'} = -2\deg(\sqrt{-D_{c'}}) \quad \text{and} \quad (\psi_*Z_A)\cdot(\psi_*Z_A) = -2\deg(\sqrt{-D_c})\deg(\psi),
\]
hence we conclude that $q = \left(\frac{\deg(\psi)\deg(\sqrt{-D_c})}{\deg(\sqrt{-D_{c'}})}\right)^{1/2} = (\deg(\psi))^{1/2}(bD_c/aD_{c'})^{1/2}$ and (a) follows. Part (b) can be obtained now from (a) by the projection formula.
\end{proof}

Let now $x \in X(K^{ab}) \subseteq \mathrm{CM}(X,K)$ be a CM point by $K$, and choose $\tilde x \in q^{-1}(x)$ any preimage of $x$ under the forgetful morphism $q: X^M \to X$ of Shimura curves. The fibre $\mathcal A_{\tilde x}$ is a QM abelian surface with $\End(\mathcal A_{\tilde x},\iota_{\tilde x}) = R_c$, for some positive integer $c$. Hence we have a well defined (up to sign) cycle $Z_{\tilde{x}} \in \CH^1(\mathcal A_{\tilde x})$. 

The cycles $Z_{\tilde x}$, for $\tilde x \in q^{-1}(x)$, can be chosen in a compatible way with respect to the action of $G(M)$. Namely, every $g\in G(M)$ extends to an automorphism $g: \mathcal A \to \mathcal A$, and induces an isomorphism $g: \mathcal A_{\tilde x} \to \mathcal A_{g(\tilde x)}$ for every $\tilde x \in q^{-1}(x)$. Then we may require that $g_*(Z_{\tilde x}) = Z_{g(\tilde x)}$ for every $\tilde{x} \in q^{-1}(x)$ and every $g \in G(M)$. 

On the other hand, if $\mathcal W$ denotes the group of Atkin--Lehner involutions acting on $X$, each $w\in \mathcal W$ extends canonically to an involution on $X^M$, which we still denote by the same symbol $w$, commuting with the action of $G(M)$. Similarly, also the group $\Gal(K_c/K)$ acts on $X^M/K_c$, and in particular on $q^{-1}(\mathrm{CM}(X,K,c))$. The actions of both $\mathcal W$ and 
$\Gal(K_c/K)$ extend to the universal abelian surface $\mathcal A$, and we may choose the cycles $Z_{\tilde x}$ so that 
\[
w_*(Z_{\tilde x}) = Z_{w(\tilde x)} \quad \delta_*(Z_{\tilde x}) = Z_{\delta(\tilde x)}
\]
for every $w\in \mathcal W$ and every $\delta \in \Gal(K_c/K)$.

Continue to fix our CM point $x\in \mathrm{CM}(X,K,c)$ and any lift $\tilde x \in q^{-1}(x)$. The inclusion $i_{\tilde x}: \pi^{-1}(\tilde x) = \mathcal A_{\tilde x} \hookrightarrow \mathcal A$ of the fibre $\pi^{-1}(\tilde x)$ into the universal abelian surface $\mathcal A$ induces a map $\CH^1(\mathcal A_{\tilde x}) \to \CH^2(\mathcal A)$, and hence we can consider the image of $Z_{\tilde x}$ in $\CH^2(\mathcal A)$. More generally, the inclusion $i_{\tilde x}^r: \pi_r^{-1}(\tilde x) = \mathcal A_{\tilde x}^r \hookrightarrow \mathcal A^r$ of the fibre of $\pi_r: \mathcal A^r \to X^M$ above $\tilde x$ into the $r$-th Kuga--Sato variety $\mathcal A^r$ over $X^M$, induces a map 
\[
(i_{\tilde x}^r)_*: \CH^r(\mathcal A_{\tilde x}^r) \, \, \longrightarrow \, \, \CH^{r+1}(\mathcal A^r),
\]
and we let $\mathcal Z_{\tilde x} := (i_{\tilde{x}}^r)_*(Z_{\tilde x}^r) \in \CH^{r+1}(\mathcal A^r/K_c)$ be the image of $Z_{\tilde{x}}^r = Z_{\tilde x} \times \cdots \times Z_{\tilde x}$. Hence $\mathcal Z_{\tilde x}$ is a cycle of codimension $r+1$ in $\mathcal A^r$ defined over $K_c$, the ring class field of conductor $c = c(x)$. Notice that for every $w\in \mathcal W$ and $\delta \in \Gal(K_c/K)$ the above compatibility conditions for the cycles $Z_{\tilde x}$ imply that 
\[
w_*(\mathcal Z_{\tilde{x}}) = \mathcal Z_{w(\tilde{x})} \quad \text{and} \quad 
\delta_*(\mathcal Z_{\tilde{x}}) = \mathcal Z_{\delta(\tilde{x})}.
\]

\section{The Euler system}

We keep assuming the Heegner hypothesis (Heeg), so that for every integer $n\geq 1$ with $\gcd(ND_K,n)=1$ the set $\CM(X,K,n)$ of CM points of conductor $n$ is non-empty. Under the maps $\Phi_{f,K_n}$ induced by the relevant $p$-adic \'etale Abel--Jacobi maps, the CM cycles constructed in the previous section give rise to a collection of cohomology classes enjoying the compatibility properties that turn them into an Euler system of Kolyvagin type.

We shall restrict ourselves to square-free integers $n\geq 1$ such that $\gcd(pND_K,n)=1$ and whose prime factors are all {\em inert} in $K$. Such primes will be referred to as {\em Kolyvagin primes}, although later we will need to impose more precise restrictions. Fix an embedding $g: K \to B$, optimal with respect to $R_1$ and $\mathcal R$, where $R_1=R_K$ stands for the ring of integers in $K$. Such an embedding defines a CM point $x(1) = [1,g] \in \CM(X,K,1)$ of conductor $1$ in $X$. For each prime $\ell \nmid pND_K$ inert in $K$, choose an element $b(\ell) \in B_{\ell}^{\times} \subset \hat B^{\times}$ satisfying $\ord_{\ell}(\mathrm n(b(\ell)))=1$, and if $n=\ell_1\cdots \ell_k$ is a product of pairwise distinct primes $\ell_i \nmid pND_K$ inert in $K$, set $b(n) := b(\ell_1)\cdots b(\ell_k) \in \hat{B}^{\times}$. Then $x(n) := [b(n),g] \in \CM(X,K,n)$ is a CM point of conductor $n$ in $X$. This way, we have defined a collection of algebraic points $x(n) \in X(K_n)$, indexed by positive integers $n$ as above.

For each $x(n)$, we choose a preimage $\tilde{x}(n) \in q^{-1}(x(n))$ in $X^M$, and for $\tilde x(n)$ we have defined a CM cycle $\mathcal Z_n := \mathcal Z_{\tilde{x}(n)} \in \CH^{r+1}(\mathcal A^r/K_n)$. We denote by $y_n \in \H^1(K_n,V)$ the image of $\mathcal Z_n$ by the $\mathbb T[\Gal(K_n/K)]$-equivariant morphism 
\[
\Phi_{f,K_n}: \CH^{r+1}(\mathcal A^r/K_n) \otimes \ZZ_p\, \, \longrightarrow \, \, \H^1(K_n,V).
\]
 Because of the averaging projector $p_G$ involved in the definition of $\Phi_{f,K_n}$, the cohomology class $y_n$ does not depend on the choice of $\tilde{x} \in q^{-1}(x)$.

\subsection{The Euler system relations}

Let $n = \ell_1\cdots \ell_k$ be the square-free product of pairwise distinct primes as above, all of them inert in $K$ and not dividing $pND_K$. We may assume for simplicity that $R_K^{\times} =\{\pm 1\}$, which amounts to say that $K\neq \QQ(\sqrt{-1}), \QQ(\sqrt{-3})$ (otherwise, all the discussion applies almost without any change, with $u_K:=|R_K^{\times}|/2$ appearing in many of the computations). If we write $G_n:=\Gal(K_n/K_1)$, then $G_n \simeq \prod_{\ell\mid n} G_{\ell}$, where each $G_{\ell}=\Gal(K_{\ell}/K_1)$ is cyclic of order $\ell+1$. We fix once and for all a generator $\sigma_{\ell}\in G_{\ell}$ for each prime $\ell$ inert in $K$. By class field theory, if $n = m\ell$ and $\lambda = (\ell)$ is the only prime in $K$ above $\ell$, then $\lambda$ splits completely in $K_m/K$, and each prime $\lambda_m$ in $K_m$ above $\lambda$ is totally ramified in $K_n/K_m$, so that $\lambda_m = (\lambda_n)^{\ell+1}$.

When $n$ varies over positive integers as above, the cohomology classes $y_n$ arising from CM cycles in $\mathcal A^r$ enjoy the following norm compatibility relations.

\begin{proposition}[Global norm compatibilities] \label{gnc}
Let $n= \ell m$ be a product of Kolyvagin primes. Then
\[
T_\ell(y_m)= \mathrm{cor}_{K_{m\ell},\ K_m}(y_{m \ell}) = a_{\ell} y_{m \ell}.
\]
\end{proposition}
\begin{proof}
We know that $T_{\ell}$ acts on $V$ as multiplication by $a_{\ell}$, thus it suffices to prove the first equality of the statement. Suppose that $y_m \in \mathrm H^1(K_m,V)$ is the image under $\Phi_{f,K_m}$ of a CM cycle $\mathcal Z_{\tilde{x}(m)} = (i_{\tilde{x}(m)}^r)_*(Z_{\tilde{x}(m)}^r)$, for some $\tilde x(m) \in X^M(\CC)$ lying above a CM point $x(m)\in X(K_m)$ of conductor $m$. The divisor $T_{\ell}(\tilde x(m))$ consists of $\ell+1$ points, lying above the $\ell+1$ points whose formal sum is $T_{\ell}(x(m))$. By compatibility of the Hecke correspondence $T_{\ell}$ acting on $X$, $X^M$ and $\mathcal A^r$, we have 
\[
 T_{\ell} \mathcal Z_{\tilde x(m)}= \sum_{\tilde x \ \in \ \mathrm{Supp}(T_{\ell}(\tilde x(m)))} \mathcal Z_{\tilde x}.
\]
But using the norm relations for CM points on $X$ (see \cite[Proposition 4.8 (ii)]{nekovar2007euler}), the right hand side equals 
\[
 \sum_{\sigma \in \Gal(K_{m \ell}/K_m) }  \mathcal Z_{\sigma(\tilde x(m\ell))} = \sum_{\sigma \in \Gal(K_{m \ell}/K_m) } \sigma_*(\mathcal Z_{\tilde x(m\ell)}),
\]
where the last equality follows from the compatibility of the CM cycles under Galois action. Finally, since $\Phi_{f,K_m}$ is Hecke- and Galois-equivariant, the above relation implies $T_{\ell}(y_m) = \mathrm{cor}_{K_n, K_m}(y_{m \ell})$, as was to be proved.
\end{proof}

\begin{proposition}[Local norm compatibilities]
Let $n= \ell m$ be a product of Kolyvagin primes. Let $\lambda_n$ be a prime dividing $\ell$ in $K_n$. Then 
\[
y_{n,\lambda_n} = \mathrm{Frob}(\ell)(\mathrm{res}_{K_{\lambda_m},K_{\lambda_n}}(y_{n,\lambda_n})) \in \H^1(K_{\lambda_n},V).
\]
\end{proposition}
\begin{proof}
Let $x(m)$ and $x(n)$ be the CM points on $X$ corresponding to the classes $y_m$ and $y_n$, respectively, and let $A_m$ be a QM-abelian surface with CM representing $x(m)$. Then $x(n)$ can be represented by  $A_n = A_m/C$, where $C\subseteq A[\ell]$ is a subgroup of order $\ell^2$, cyclic as $\mathcal R$-submodule, and $A_m$ and $A_n$ are related by the canonical isogeny $A_m \to A_m/C$. Since $\ell$ is inert in $K$, the reductions of $A_m$ and $A_n$ modulo $\lambda_m$ and $\lambda_n$, respectively, are both products of two supersingular elliptic curves (recall that $\lambda_m$ is totally ramified in $K_n/K_m$, hence the residue fields of both $K_m$ at $\lambda_m$ and $K_n$ at $\lambda_n$ coincide, and are in fact isomorphic to the finite field of $\ell^2$ elements). Then the isogeny $A_m \to A_m/C$ reduces to Frobenius on each factor and the claim follows by the relations between CM cycles under isogeny in Proposition \ref{CMcycles-isogenies} (the constants $a$ and $b$ are forced to be equal, and $c'=c\ell$).
\end{proof}

The cohomology classes $y_n$ also enjoy the following compatibility with respect to complex conjugation, which stems directly from the action of complex conjugation on the CM cycles.

\begin{proposition}
If $\rho$ denotes complex conjugation, then 
\[
\rho(y_n) = -\varepsilon \sigma(y_n)
\]
for some $\sigma \in \Gal(K_n/K)$, where $\varepsilon$ denotes the sign in the functional equation of $L(f,s)$.
\end{proposition}
\begin{proof}
Let $Z_x \subseteq \mathcal A_{\tilde x}$ be the cycle associated to (a lift $\tilde x \in X^M$ of) the CM point $x=x(n)$ of conductor $n$. Then $\rho(Z_x) = - Z_{\rho(x)}$.

Besides, it is well-known that $\rho(x) = W_N(\sigma(x))$ for some $\sigma\in \Gal(K_n/K)$, where $W_N$ stands for the Atkin--Lehner involution associated with $N=N^+N^-$ acting on $X$ (see \cite[p. 135]{TianThesis}). By using the compatibility of the CM cycles with the actions of Galois and $\mathcal W$, and that these actions commute, it follows that 
\[
\rho(\mathcal Z_n) = (-1)^rW_{N,*}(\sigma_*(\mathcal Z_n)) = (-1)^r\sigma_*(W_{N,*}(\mathcal Z_n)).
\]

Finally, since $\Phi_{f,K_n}$ is Hecke- and Galois-equivariant, we deduce that 
\[
\rho(\Phi_{f,K_n}(\mathcal Z_n)) = (-1)^r w_{N^+,N^-}\sigma(\Phi_{f,K_n}(\mathcal Z_n)),
\]
which is equivalent to the relation we want, since $\varepsilon = (-1)^{r+1} w_{N^+,N^-}$.
\end{proof}

\subsection{Kolyvagin cohomology classes}

Recall that $V$ is a free $\mathcal O_F \otimes \ZZ_p$-module of rank $2$, where $\mathcal O_F$ stands for the ring of integers in the number field $F$ generated by the Hecke eigenvalues $a_n=a_n(f_{\infty})$. If $\mathcal O_{\wp}$ denotes the completion of $\mathcal O_F$ at a prime $\wp$ of $F$ above $p$, then there is a canonical decomposition $\mathcal O_F \otimes \ZZ_p = \oplus \mathcal O_{\wp}$, where the sum is over all such primes $\wp \mid p$. Fix once and for all a prime $\wp \mid p$ of $F$. Then $V_{\wp} := V \otimes_{\mathcal O_F \otimes \ZZ_p} \mathcal O_{\wp}$ is a free $\mathcal O_{\wp}$-module of rank $2$, and there are natural localization morphisms 
\[
\H^1(K_n,V) \longrightarrow \H^1(K_n,V_{\wp}), \quad y_n \longmapsto y_{n,\wp}.
\] 
From now on, we write $Y := V_{\wp}\otimes \QQ_p/\ZZ_p$ and $Y_s := Y_{p^s}$ for every $s\geq 1$, hence $Y_s = V_{\wp}/p^sV_{\wp}$ for $s\geq 1$. We remark that for the sake of simplicity, we did not consider these integral and mod $p^s$ representations in the Introduction, and rather stayed at the level of $F_{\wp}$-vector spaces. Indeed, the representation $V_{\wp}(f)$ of the Introduction corresponds to the $F_{\wp}$-vector space $V_{\wp} \otimes_{\mathcal O_{\wp}} F_{\wp} = V_{\wp} \otimes_{\ZZ_p} \QQ_p$. For each $s\geq 1$, we have a natural reduction map  
\[
\mathrm{red}_s: \H^1(K_n,V_{\wp}) \longrightarrow \H^1(K_n,Y_s)
\]
(and all such maps are compatible in the natural way when $s$ varies), and we denote by $L_s/K$ the Galois extension of $K$ cut out by the representation $Y_s$.

It follows directly from \cite[Proposition 6.3]{nekovar1992kolyvagin} that $Y_s^{\Gal(\bar{\QQ}/K_n)} = Y_s^{\Gal(\bar{\QQ}/K_1)}$ for every square-free integer $n$ which is a product of primes as above, and further there exists an integer $s_1\geq 0$, which does {\em not} depend on $s$, such that $Y_s^{\Gal(\bar{\QQ}/K_1)}$ (hence also $Y_s^{\Gal(\bar{\QQ}/K_n)}$, for all $n$) is killed by $p^{s_1}$. We state below a direct consequence of this (cf. \cite[Corollary 6.4]{nekovar1992kolyvagin}), for later reference.

\begin{corollary}\label{s1kercoker}
There is an integer $s_1 \geq 0$, independent from $s$, such that both the kernel and cokernel of the restriction map
\[
\mathrm{res}_{K_1,K_n}: \H^1(K_1,Y_s) \, \, \longrightarrow \, \, \H^1(K_n,Y_s)^{G_n}
\]
are killed by $p^{s_1}$.
\end{corollary}

Next we want to construct $G_n$-invariant classes in $\H^1(K_n,Y_s)$ starting from the localized classes $y_{n,\wp}$. To do so, we first need to restrict a bit more the indices $n$ that will be admissible in our system of classes, or rather on the primes that will be allowed as factors of $n$.

\begin{definition}
For each $s\geq 1$, we define $\Sigma_1(s)$ to be the set of rational primes $\ell\nmid 2pND_K$ such that the conjugacy class $\Frob_{\ell}(K_cL_s/\QQ)$ of the arithmetic Frobenius automorphism $\Frob_{\ell}$ in $\Gal(K_cL_s/\QQ)$ coincides with the conjugacy class $\Frob_{\infty}(K_cL_s/\QQ)$ of complex conjugation. We then put, for every $k\geq 1$,
\[
\Sigma_k(s) = \{ \ell_1\cdots \ell_k: \ell_i \in \Sigma_1(s) \text{ pairwise distinct}\}.
\]
Primes in $\Sigma_1(s)$ will be referred to as {\em $s$-Kolyvagin primes}, or just {\em Kolyvagin primes}.
\end{definition}

The condition on the conjugacy classes of Frobenius and complex conjugation can be rephrased by saying that $\Frob_{\ell}$ and the complex conjugation $\Frob_{\infty}$ have the same characteristic polynomial modulo $p^s$, that is,
\[
 x^2-1 \equiv \ell x^2 - \frac{a_{\ell}}{\ell^r} x + 1 \pmod{p^s},
\]
This is equivalent to the assertion that 
\[
\left(\frac{-D_K}{\ell}\right) = -1 \quad \text{and} \quad a_{\ell} \equiv \ell + 1 \equiv 0 \pmod{p^s}.
\]
In particular, notice that primes $\ell\in \Sigma_1(s)$ are inert in $K$.

Let $\ell \in \Sigma_1(s)$ be a Kolyvagin prime, and recall that we denoted by $\sigma_{\ell}$ a fixed generator $G_{\ell} = \Gal(K_\ell/K_1)$. Recall also {\em Kolyvagin's trace} and {\em derivative} operators
\[
\Tr_\ell = \sum_{i=0}^{|G_\ell|-1} \sigma_\ell^i, \quad D_\ell = \sum_{i=0}^{|G_\ell| -1} i \sigma_\ell^i \, \, \in \, \, \ZZ[G_\ell],
\]
related by the identities
\[
(\sigma_\ell -1)D_\ell = |G_\ell| - \Tr_\ell = \ell + 1 - \Tr_\ell.
\]
If $n= \ell_1 \cdots \ell_k \in \Sigma_k(s)$, then one also defines
\[
D_n := D_{\ell_1} \cdots D_{\ell_k} \, \in \, \ZZ[G_n].
\]

\begin{proposition}
The cohomology class $D_n\mathrm{red}_s(y_{n,\wp}) \in \H^1(K_n,Y_s)$ is $G_n$-invariant, i.e. it belongs to $\H^1(K_n,Y_s)^{G_n}$.
\end{proposition}
\begin{proof}
Let $\ell$ be a (Kolyvagin) prime dividing $n$ and set $n=m\ell$. Then we have 
\[
(\sigma_{\ell}-1)D_n y_{n,\wp} = D_{m}(\ell+1-\Tr_{\ell})y_{n,\wp} =  D_{m}(\ell+1)y_{n,\wp} - a_{\ell}D_m y_{m,\wp},
\]
where in the second equality we use $\mathrm{res}_{K_m,K_n}\circ\mathrm{cor}_{K_n,K_m} = \Tr_{\ell}$ and Proposition \ref{gnc}. Now since $\ell\in \Sigma_1(s)$ we know that $\ell+1 \equiv a_{\ell} \equiv 0 \pmod{p^s}$, hence the statement follows.
\end{proof}

By virtue of Corollary \ref{s1kercoker}, this proposition implies that, possibly up to multiplying by $p^{s_1}$ (where recall that $s_1$ is independent of $s$), the derived classes $D_n\mathrm{red}_s(y_{n,\wp}) \in \H^1(K_n,Y_s)$ can be lifted to $\H^1(K_1,Y_s)$. This lifting is often referred to as ``Kolyvagin's corestriction'', and is reviewed in detail in \cite[Section 7]{nekovar1992kolyvagin}.

More precisely, continue to fix $s \gg 0$, put $s' = s+s_1$ and require Kolyvagin primes to lie in $\Sigma_1(s')$ (so that $a_{\ell}\equiv \ell + 1 \equiv 0 \pmod{p^{s'}}$ for all primes $\ell$ dividing $n$). Multiplication by $p^{s_1}$ induces a homomorphism $j: Y_{s'} \to Y_{s}$. A system of cohomology classes 
\[
\kappa_s(n) \in \H^1(K,Y_s)
\]
can be defined in the following way. For $n=1$, just set $\kappa_s(1):=\mathrm{cor}_{K_1,K}(\mathrm{red}_s(y_{1,\wp}))$. If $n=\ell$ is a Kolyvagin prime, then $p^{s_1}D_{\ell}\mathrm{red}_{s'}(y_{\ell,\wp})=\mathrm{res}_{K_1,K_{\ell}}(z_{\ell})$ for some class $z_{\ell}\in \H^1(K_1,Y_{s'})$, and two choices of $z_{\ell}$ differ by an element in the image of $\H^1(K,Y_{s_1}) \to \H^1(K,Y_{s'})$, and hence in the kernel of $j_{*}$. In view of this, the class
\[
\kappa_s(\ell) := \mathrm{cor}_{K_1,K}(j_*(z_{\ell})) \in \H^1(K,Y_s),
\]
is well-defined and does not depend on the choice of $z_{\ell}$. More generally, if $n=\ell_1\cdots \ell_k$ is a (square-free) product of Kolyvagin primes, then $p^{s_1}D_n\mathrm{red}_{s'}(y_{n,\wp}) = \mathrm{res}_{K_1,K_{\ell}}(z_n)$ for some $z_n \in \H^1(K_1,Y_{s'})$. As before, the class
\[
\kappa_s(n) := \mathrm{cor}_{K_1,K}(j_*(z_n)) \in \H^1(K,Y_s),
\]
is well-defined.

\subsection{Localization away from $p$}

Having defined the cohomology classes $\kappa_s(n)$, we end this section by describing their local behaviour at places $v$ of $K$ not dividing $p$. Of special interest are the localizations at places above the prime factors of $n$. We may start by fixing some notation. Let $\ell\in \Sigma_1(s)$ be a Kolyvagin prime dividing $n$, write $n = m\ell$ and let $\lambda$ be the (unique) prime of $K$ above $\ell$. We fix a prime $\lambda_n$ of $K_n$ above $\ell$, which uniquely determines primes $\lambda_m$, $\lambda_{\ell}$ and $\lambda_1$ of $K_m$, $K_{\ell}$ and $K_1$, respectively, all of them over $\ell$. Recall that $\lambda=(\ell)$ splits completely in the extension $K_m/K$, whereas $\lambda_m$ is totally ramified in $K_n/K_m$, hence $\lambda_m = (\lambda_n)^{\ell+1}$. At the level of completions, we have $K_{\lambda_n} = K_{\lambda_{\ell}}$ and $K_{\lambda_m}=K_{\lambda_1}=K_{\lambda}$, and our choice of $\lambda_n$ determines also an isomorphism
\[
\Gal(K_{\lambda_{\ell}}/K_{\lambda}) = \Gal(K_{\lambda_n}/K_{\lambda_m}) \simeq \Gal(K_{\ell}/K_1) = \langle \sigma_{\ell}\rangle.
\]

In particular, the choice of $\lambda_n$ identifies the generator $\sigma_{\ell}$ with an element of $\Gal(K_{\lambda_{\ell}}/K_{\lambda})$. Such an element can be lifted to a generator $\tau_{\ell}$ of $\Gal(K_{\lambda}^{\mathrm{tr}}/K_{\lambda}^{\ur}) \simeq \hat{\ZZ}^{(\ell)}(1)$, where $K_{\lambda}^{\mathrm{tr}}$ is the maximal tamely ramified extension of $K_{\lambda}$ and $\hat{\ZZ}^{(\ell)} = \prod_{q\neq\ell} \ZZ_q$. This lift is well-defined modulo $(\ell+1)\hat{\ZZ}^{(\ell)}(1)$, and under the canonical projection $\hat{\ZZ}^{(\ell)}(1) \to \mu_{p^{s'}}$ it is sent to some primitive $p^{s'}$-th root of unity, say $\zeta_{\lambda,s'} \in \mu_{p^{s'}}(K_{\lambda})$. Tame duality then yields (cf. \cite[Proposition 8.1]{nekovar1992kolyvagin}) $\mathcal O_{\wp}$-linear canonical isomorphisms
\begin{equation}\label{def-alpha}
\alpha_{\lambda,s'}: \H^1_{\mathrm{ur}}(K_{\lambda},Y_{s'}) \stackrel{\simeq}{\longrightarrow} Y_{s'}(K_{\lambda}),
\end{equation}
\begin{equation}\label{def-beta}
\beta_{\lambda,s'}: \H^1(K_{\lambda}^{\mathrm{ur}},Y_{s'}) \stackrel{\simeq}{\longrightarrow} \mathrm{Hom}(\mu_{p^{s'}}(K_{\lambda}),Y_{s'}(K_{\lambda})) \simeq Y_{s'}(K_{\lambda}),
\end{equation}
with $\beta_{\lambda,s'}$ being evaluation at the root of unity $\zeta_{\lambda,s'}$, and a perfect pairing 
\begin{equation}\label{def-pairing-ur}
\langle , \rangle_{\lambda,s'}: \H^1_{\mathrm{ur}}(K_{\lambda},Y_{s'}) \times \mathrm H^1(K_{\lambda}^{\mathrm{ur}},Y_{s'}') \, \, \longrightarrow \, \, \ZZ/p^{s'}\ZZ,
\end{equation}
where $Y_{s'}' = \Hom(Y_{s'},\mu_{p^{s'}})$. Further, the isomorphism 
\begin{equation}\label{def-phi-ur}
\phi_{\lambda, s'} = \beta_{\lambda,s'}^{-1} \circ \alpha_{\lambda,s'}:  \mathrm H^1_{\mathrm{ur}}(K_{\lambda},Y_{s'}) \stackrel{\simeq}{\longrightarrow} \mathrm H^1(K_{\lambda}^{\mathrm{ur}},Y_{s'})
\end{equation}
interchanges cocycles with the same values on $\Frob(\ell)$ and $\tau_{\ell} \pmod{p^{s'}}$. After identifying $Y_{s'}$ with its dual $Y_{s'}'$ via the pairing $[,]_{s'}$ from Proposition \ref{prop:V}, the pairing $\langle , \rangle_{\lambda,s'}$ satisfies the relation
\begin{equation}\label{eqn:pairings_zeta}
\zeta_{\lambda,s'}^{\langle x, \phi_{\lambda,s'}(y)\rangle_{\lambda,s'}} = [\alpha_{\lambda,s'}(x),\alpha_{\lambda,s'}(y)]_{s'}.
\end{equation}

Finally, localizing the inflation-restriction sequence for $K_n/K_1$ yields a canonical splitting
\[
\mathrm H^1(K_{\lambda},Y_{s'}) \simeq \mathrm H^1_{\mathrm{ur}}(K_{\lambda},Y_{s'}) \oplus \mathrm H^1(K_{\lambda}^{\mathrm{ur}},Y_{s'}).
\]
Both factors in this splitting are isomorphic to $Y_{s'}(K_{\lambda})$, via the canonical $\mathcal O_{\wp}$-linear isomorphisms $\alpha_{\lambda,s'}$ and $\beta_{\lambda,s'}$, respectively.

On the other hand, complex conjugation $\rho \in \Gal(K/\QQ) = \Gal(K_{\lambda}/\QQ_{\ell})$ acts naturally on several groups involved in our discussion. We denote by a superscript $\pm$ the corresponding $(\pm)$-eigenspaces for this action. Notice that $\mu_{p^{s'}}(K_{\lambda}) = \mu_{p^{s'}}(K_{\lambda})^-$ by the assumptions on $\ell$, and each of the eigenspaces $Y_{s'}(K_{\lambda})^{\pm}$ is a free $\mathcal O_{\wp}/p^{s'}$-module of rank $1$. Since the local pairing $\langle , \rangle_{\lambda, s'}$ in \eqref{def-pairing-ur} is $\rho$-equivariant, it induces non-degenerate pairings
\[
\langle , \rangle_{\lambda,s'}^{\pm}: \H^1_{\ur}(K_{\lambda},Y_{s'})^{\pm} \times \H^1(K_{\lambda}^{\ur},Y_{s'})^{\pm} \, \, \longrightarrow \, \, \ZZ/p^{s'}\ZZ.
\]
In contrast, the isomorphism $\phi_{\lambda, s'}$ is $\rho$-antiequivariant, and therefore it induces isomorphisms
\[
\H^1_{\ur}(K_{\lambda},Y_{s'})^{\pm} \simeq \H^1(K_{\lambda}^{\ur},Y_{s'})^{\mp}.
\]

The next proposition summarizes the relevant properties of the localizations of the Kolyvagin cohomology classes $\kappa_s(n)$ at places of $K$ outside of $p$.

\begin{proposition}\label{localizationkappas}
Let $v$ be a non-archimedean place of $K$ and $n$ be a product of Kolyvagin primes. 
\begin{itemize}
\item[i)] $\kappa_s(n) \in \H^1(K,Y_s)^{\varepsilon_{n}}$, where $\varepsilon_n = (-1)^{n-1}\varepsilon$.
\item[ii)] If $v\nmid Nnp$, then $\kappa_s(n)_v \in \H^1_{\ur}(K_v, Y_s)$.
\item[iii)] There exists a constant $s_2$ such that $p^{s_2}\H^1(K_v,V/p^sV)$ for all places $v\mid N$ of $K$ and all $s\geq 0$. In particular, if $v\mid N$ then $p^{s_2}\kappa_s(n)_v = 0$.
\item[iv)] If $n=m\ell$ and $\lambda =(\ell)$ is the only prime of $K$ above $\ell$, then 
\[
\left(\frac{(-1)^r \varepsilon_n a_{\ell}-(\ell+1)}{p^{s'}}\right) \kappa_s(n)_{\ell} = \left(\frac{(\ell+1)\varepsilon_n - a_{\ell}}{p^{s'}}\right) p^{ds_1} \phi_{\lambda,s}(\kappa_s(m)_{\lambda}),
\]
where $d=1$ if $n$ is a product of two primes, and $d=0$ otherwise. If both $((\ell+1) \pm a_{\ell})/p^{s'}$ are units in $\mathcal O_{\wp}$, then the above relation simplifies to 
\[
\kappa_s(n)_{\lambda} = u_{\ell,\varepsilon_{\ell}}p^{ds_1} \phi_{\lambda,s}(\kappa_s(m)_{\lambda}), \quad u_{\ell,\varepsilon_{\ell}} \in (\mathcal O_{\wp}/p^s)^{\times}.
\]
\end{itemize}
\end{proposition}
\begin{proof}
The first assertion is a direct consequence of the action of $\rho$ on CM cycles (and hence, on the classes $y_n$) and the relation with Kolyvagin's derivative operator, namely $\rho D_n = (-1)^n D_n \rho$. Statement ii) is clear since both $K_n/K$ and $y_n$ are unramified at the place $v$, and iii) follows from \cite[Lemma 10.1]{nekovar1992kolyvagin}. Finally, iv) is obtained by applying Nekov\'a\v r's discussion on localization of Kolyvagin's corestriction in \cite[Section 9]{nekovar1992kolyvagin} (see also \cite[Proposition 10.2 (4)]{nekovar1992kolyvagin}).
\end{proof}

\begin{corollary}\label{cor:pairing}
If both $\ell+1 \pm a_{\ell}$ divide $p^{s'+e}$ in $\mathcal O_{\wp}$, then for all $t_{\lambda} \in \H^1_{\ur}(K_{\lambda},Y_s)$
\[
\zeta_{\lambda,s}^{\langle t_{\lambda}, p^e\kappa_s(n)_{\lambda}\rangle_{\lambda,s}} = [\alpha_{\lambda,s}(t_{\lambda}), u_{\ell,\varepsilon_n}p^{ds_1+e}\alpha_{\lambda,s}(\kappa_s(n/\ell))]_s,
\]
where $\zeta_{\lambda,s} = \zeta_{\lambda,s'}^{p^{s_1}}$.
\end{corollary}

\section{The Selmer group}\label{sec:Selmer}

So far, we have seen that, possibly up to multiplying by $p^{s_2}$, the cohomology classes $\kappa_s(n)$ are unramified at every place of $K$ not dividing $np$. Further, their localizations at the primes of $K$ dividing $n$ are subject to the relations in Proposition \ref{localizationkappas} iv).

Now if $v$ is a place of $K$ above $p$, then the $\QQ_p$-vector space $W := V \otimes \QQ$ is equipped with a continuous $\Gal(\bar K_v/K_v)$-action, and following Bloch and Kato it is customary to set 
\[
\H^1_{\texttt f}(K_v,W) := \mathrm{ker}(\H^1(K_v,W) \to \H^1(K_v,W \otimes B_{\mathrm{cris}}))
\]
and 
\[
\H^1_{\texttt g}(K_v,W) := \mathrm{ker}(\H^1(K_v,W) \to \H^1(K_v,W \otimes B_{\mathrm{dR}})),
\]
where $B_{\mathrm{cris}}$ and $B_{\mathrm{dR}}$ are Fontaine's period rings. In order to deal with the representations $V$ and $Y_s$, if $? \in \{\texttt f, \texttt g\}$ we denote by $\mathrm H^1_?(K_v,V) \subseteq \mathrm H^1(K_v,V)$ the preimage of $\mathrm H^1_?(K_v,W)$ under the natural homomorphism $\H^1(K_v,V) \to \H^1(K_v,W)$, and by 
$\H^1_?(K_v,Y_s) \subseteq \H^1(K_v,Y_s)$ the image of $\H^1_?(K_v,V)$ under the natural reduction homomorphism 
$\H^1(K_v,V) \to \H^1(K_v,Y_s)$.

As in \cite[Lemma 11.1]{nekovar1992kolyvagin}, the fact that $V\otimes \QQ$ is crystalline implies that if $v$ is a prime of $K$ dividing $p$ and $K'/K_v$ is any finite extension, then $\H^1_{\mathrm{\texttt f}}(K',V) = \H^1_{\mathrm{\texttt g}}(K',V)$ and the Abel--Jacobi map over $K'$ factors through $\H^1_{\mathrm{\texttt f}}(K',V)$. In particular, since $\H^1_{\mathrm{\texttt f}}(K_v,Y_s)$ depends only on the action of the inertia subgroup of $\Gal(\bar K_v/K_v)$, for every square-free product of Kolyvagin primes $n$ and any prime $v$ of $K$ above $p$ it follows that $\kappa_s(n)_v \in \H^1_{\mathrm{\texttt f}}(K_v,Y_s)$ (because $K_n/K$ is unramified at $v$). This leads naturally to the definition of the ($p^s$-th) Selmer group:

\begin{definition}
The ($p^s$-th) {\em Selmer group} $\Sel^{(s)}_{\wp}(f,K) \subseteq \H^1(K,Y_s)$ is defined as 
\[
\Sel^{(s)}_{\wp}(f,K):= \{x \in \H^1(K,Y_s): x_v \in \H^1_{\mathrm{ur}}(K_v,Y_s) \text{ for all } v \nmid Np \text{ and } x_v \in \H^1_{\texttt f}(K_v,Y_s) \text{ for } v \mid p\}.
\]
\end{definition}

If $v$ is a place of $K$ not dividing $N$, then $\mathcal A$ has good reduction at $v$, and therefore we infer from \cite[Lemma 4.1]{nekovar1992kolyvagin} that $\H^1(K_v,V)$ consists only of unramified classes. Hence from the very definition of $\Sel^{(s)}_{\wp}(f,K)$ we see that the global Abel--Jacobi map from \eqref{AJpKugaSato} factors through
\begin{equation}\label{AJSel}
\mathrm{CH}^{r+1}(\mathcal A^r/K)_0 \otimes \mathcal O_{\wp}/p^s\mathcal O_{\wp} \, \, \longrightarrow \, \, \Sel^{(s)}_{\wp}(f,K).
\end{equation}

On the other hand, given arbitrary classes $x, y \in \mathrm H^1(K,Y_s)$ the reciprocity law asserts that 
\[
\sum_v \langle x_v, y_v\rangle_{v,s} = 0 \quad \text{in} \, \, \ZZ/p^s\ZZ,
\]
where the sum is over all the places in $K$. This is actually a finite sum, since $\langle x_v, y_v\rangle_{v,s}$ vanishes for every place $v$ for which both $x$ and $y$ are unramified. If $n$ is a product of Kolyvagin primes, then the cohomology classes $\kappa_s(n)$ are unramified at all places not dividing $pn$, possibly after multiplying by $p^{s_2}$, and we also know that $\kappa_s(n)_v \in \H^1_{\mathrm{\texttt f}}(K_v,Y_s)$ for every place $v$ of $K$ above $p$. But the finite part $\H^1_{\mathrm{\texttt f}}(K_v,Y_s)$ is isotropic inside $\H^1(K_v,Y_s)$ at all places $v$ dividing $p$ (see \cite[Prop. 3.8]{BK-Lfunctions}), hence the above reciprocity law implies that 
\begin{equation}\label{eq:RLSel}
p^{s_2}\sum_{\ell \mid n} \langle x_{\lambda}, \kappa_s(n)_{\lambda}\rangle_{\lambda,s} = 0 \quad \text{in } \ZZ/p^s\ZZ
\end{equation}
for every $x \in \Sel^{(s)}_{\wp}(f,K)$, where for each Kolyvagin prime $\ell \mid n$ in the sum $\lambda$ denotes the unique prime of $K$ above $\ell$.

Finally, we denote
\begin{equation} \label{SelInfinity}
\Sel^{(\infty)}_{\wp}(f,K):= \varprojlim \Sel^{(s)}_{\wp}(f,K).
\end{equation}
By considering the inductive limit of the Abel--Jacobi maps \eqref{AJSel} one obtains a map 
\[
\Phi: \mathrm{CH}^{r+1}(\mathcal A^r/K)_0 \otimes \mathcal O_{\wp} \, \, \longrightarrow \, \, \Sel^{(\infty)}_{\wp}(f,K) \subseteq \H^1(K,V_{\wp}).
\]
Its cokernel is by definiton the $\wp$-primary part of the Shafarevich--Tate group, 
\[
\Sh_{\wp^{\infty}} := \mathrm{coker}(\Phi) = \Sel^{(\infty)}_{\wp}(f,K)/\mathrm{Im}(\Phi).
\]

\section{Main result}

Recall our initial setting, in which $f_{\infty} \in S_{2r+2}^{\mathrm{new}}(\Gamma_0(N))$ is assumed to be a newform of weight $2r+2\geq 4$ and level $\Gamma_0(N)$, and let $p$ be an odd prime not dividing $N\cdot(2r)!$. We write $F$ for the number field generated by the Fourier coefficients of $f_{\infty}$, $\mathcal O_F$ for its ring of integers, and fix a prime $\wp$ of $F$ above $p$. Let $N=N^+N^-$ be a factorization such that $\gcd(N^+,N^-)=1$ and $N^->1$ is the square-free product of an even number of primes, and let $K$ be an imaginary quadratic field satisfying the Heegner hypothesis spelled out in (Heeg).

The Galois representation associated to $f_{\infty}$ might be realized as a factor in the middle \'etale cohomology of the $r$-th Kuga--Sato variety $\mathcal A^r$ over the Shimura curve $X=X_{N^+,N^-}$, by using the Jacquet--Langlands correspondence to lift $f_{\infty}$ to a modular form $f$ on $X$ and following previous work of Besser and Iovita--Spiess (cf. Section \ref{sec:repns}). We have previously denoted this representation by $V = V(f) \simeq V(f_{\infty})$. It is a free $\mathcal O_F \otimes \ZZ_p$-module of rank $2$, and our choice of $\wp$ singles out a localization $V_{\wp}$.

We have seen that the Abel--Jacobi map induces a Hecke- and Galois-equivariant map 
\[
\Phi: \CH^{r+1}(\mathcal A^r/K)_0 \otimes \mathcal O_{\wp} \, \, \longrightarrow \, \, \Sel^{(\infty)}_{\wp}(f,K) \subseteq \H^1(K,V_{\wp}),
\]
by localizing at $\wp$ and projecting on the $f$-isotypical component. Defining 
\[
y_0 := \cor_{K_1/K}(y_{1,\wp}),
\]
the main result we prove in this note reads as follows:

\begin{theorem}\label{mainthm}
With the above notations, suppose $y_0$ is non-torsion. Then $\mathrm{Im}(\Phi)\otimes \QQ$ has rank $1$ and $\Sh_{\wp^{\infty}}$ is finite. More precisely, we have
\[
(\mathrm{Im}(\Phi)\otimes \QQ)^{\varepsilon} = 0 \quad \text{and} \quad (\mathrm{Im}(\Phi)\otimes \QQ)^{-\varepsilon} = F_{\wp}\cdot y_0.
\]
\end{theorem}

As already commented, the proof of this result follows Kolyvagin's method as generalized by Nekov\'a\v r in \cite{nekovar1992kolyvagin}. Indeed, once we have constructed the Euler system of CM cycles on the Kuga--Sato variety $\mathcal A^r$ and have proved the compatibility properties that the associated system of Kolyvagin cohomology classes satisfies, the proof is formally the same. In spite of this, we summarize below the argument for the convenience of the reader.

Before entering into the proof, we shall make some global observations that complement our local discussions in the previous section. Keep the same notations as before, and write $L = K(Y_{s'})$ for the Galois extension of $K$ trivializing $Y_{s'}$, $s' = s + s_1$. Let also $\zeta_{s'} \in \mu_{p^{s'}}(L)$ be a primitive $p^{s'}$-th root of unity. For each Kolyvagin prime $\ell\in \Sigma_1(s')$, we might choose a place $\lambda_L$ of $L$ above $\ell$ such that $\zeta_{s'}$ maps to $\zeta_{\lambda,s'}$ under the embedding $L \hookrightarrow L_{\lambda_L} = K_{\lambda}$. Then we put $\zeta_s := (\zeta_{s'})^{p^{s_1}}$. Under this choice of $\lambda_L$, we identify $Y_s(K_{\lambda}) \simeq Y_s(L_{\lambda_L}) = Y_s(L)$. Further, we consider the maps 
\[
\alpha_{\lambda_L,s}: \H^1_{\ur}(L_{\lambda_L},Y_s) \stackrel{\simeq}{\longrightarrow} Y_s(L_{\lambda_L}), \quad
\phi_{\lambda_L,s}: \H^1_{\ur}(L_{\lambda_L},Y_s) \stackrel{\simeq}{\longrightarrow} \H^1(L_{\lambda_L}^{\ur},Y_s)
\]
analogous to the maps $\alpha_{\lambda,s}$ and $\phi_{\lambda,s}$ introduced in \eqref{def-alpha} and \eqref{def-phi-ur}, respectively, for $L_{\lambda_L}$. And by a slight abuse of notation, we also write $\alpha_{\lambda_L,s}$ for the composition
\begin{equation}\label{def-alpha-global}
\H^1(L,Y_s) \, \, \longrightarrow \, \, \H^1(L_{\lambda_L},Y_s) \, \, \longrightarrow \, \, \H^1_{\ur}(L_{\lambda_L},Y_s) \, \, \stackrel{\alpha_{\lambda_L,s}}{\longrightarrow} Y_s(L_{\lambda}) = Y_s(L),
\end{equation}
where the first arrow is localization at $\lambda_L$ and the second one is projection on the unramified part. The composition of these maps is the evaluation at $\Frob(\lambda_L)$.

Consider the restriction map
\[
\res = \res_{K,L}: \H^1(K,Y_s) \, \, \longrightarrow \, \, \H^1(L,Y_s)^{\Gal(L/K)} = \Hom_{\Gal(L/K)}(\Gal(\bar{\QQ}/L),Y_s(L)).
\]
The formula \eqref{eqn:pairings_zeta} relating the pairings $\langle , \rangle_{\lambda,s}$ and $[,]_s$ through the root of unity $\zeta_s$ admits the following global version:

\begin{proposition}\label{prop:pairingsglobal}
Given classes $\eta, \theta \in \H^1(K,Y_s)$ such that $\eta_{\lambda}, \theta_{\lambda} \in \H^1_{\ur}(K_{\lambda},Y_s)$,
\[
\zeta_s^{\langle \eta ,\phi_{\lambda_L,s}(\theta) \rangle_{\lambda,s}} = [ \alpha_{\lambda_L,s}(\res(\eta)), \alpha_{\lambda_L,s}(\res(\theta)) ]_s.
\]
\end{proposition}

Now we are finally in position to prove Theorem \ref{mainthm}.

\begin{proof}[Proof of Theorem \ref{mainthm}]
We keep the notations as above. In particular, $s\gg 0$ is a sufficiently large integer, and $s' = s+ s_1$. The Selmer group $\Sel^{(\infty)}_{\wp}(f,K) = \varprojlim \Sel^{(s)}_{\wp}(f,K)$ is finitely generated as a $\ZZ_p$-module, and our goal is to bound it. We continue to denote $L=K(Y_{s'})$. 

Let $T = \res(\Sel^{(s)}_{\wp}(f,K)) \subseteq \H^1(L,Y_s)$ be the image of the $p^s$-th Selmer group under the restriction map, and write also $u_s(n) := \res(\kappa_s(n)) \in \H^1(L,Y_s)$ for the image of the $n$-th Kolyvagin cohomology class under restriction. The action of $\rho$ on $T$ defines two eigenspaces, $T^{\pm}$, and we will obtain our bound for $T$ (and hence eventually for $\Sel^{(\infty)}_{\wp}(f,K)$) by looking separately at $T^{\varepsilon}$ and $T^{-\varepsilon}$.

Let $L_T \subseteq L^{ab}$ denote the subfield fixed by the annihilator of $T$ under the evaluation pairing
$T \times \Gal(L^{ab}/L) \to Y_s(L)$, and put $G_T := \Gal(L_T/L)$. Then one has an induced $\Gal(L/\QQ)$-equivariant pairing 
\[
T \times G_T \, \, \longrightarrow \, \, Y_s(L), \quad (t,g) \, \, \longmapsto \, \, t(g),
\]
with the action of $\Gal(L/\QQ)$ on $T$ factoring through $\Gal(K/\QQ)$. In particular, this naturally induces a $\Gal(L/\QQ)$-equivariant map $G_T \, \, \hookrightarrow \, \, \Hom(T,Y_s(L))$ and a $\rho$-equivariant map $T \hookrightarrow \Hom_{\Gal(L/K)}(G_T,Y_s(L))$, both of them injective.

As in \cite[Section 12]{nekovar1992kolyvagin} (specifically Proposition 12.2 therein), there exist integers $a, b \geq 0$ such that, for all $s$ large enough, the following assertions hold:
\begin{itemize}
\item[(i)] $p^a\H^1(K(Y_{s'})/K,Y_s) = 0$;
\item[(ii)] $L_T \cap K(Y_{\infty}) \subseteq K(Y_{s'+a})$;
\item[(iii)] for each $g\in G_T^+$, there are infinitely many primes $\ell$ which are inert in $K$ and such that 
$\Frob_{L_T/K}(\lambda) = g$, $p^{s'}\mid \ell+1\pm a_{\ell}$ and $p^{s'+a+1} \nmid \ell+1\pm a_{\ell}$;
\item[(iv)] $p^b\mathrm{coker}(j: G_T \hookrightarrow \Hom(T,Y_s)) = 0$.
\end{itemize}

If $x$ is an element in an abelian group $A$, let $\exp(x)$ be the smallest $m\geq 0$ such that $p^mx = 0$. In the same fashion, $\exp(A)$ denotes the smallest $m\geq 0$ with $p^mA=0$. For instance, $\exp(\kappa_s(1)) = s - s_0$ and $\exp(u_s(n)) \geq \exp(\kappa_s(n)) - a$.

Fix an element $\psi_{\varepsilon} \in \Hom(T^{\varepsilon},Y_s^{\varepsilon})$ of maximal exponent, i.e. such that 
\[
\exp(\psi_{\varepsilon}) = \exp(\Hom(T^{\varepsilon},Y_s^{\varepsilon})),
\]
and notice that this exponent also equals $\exp(T^{\varepsilon})$. Also choose $\psi_{-\varepsilon} \in \Hom(T^{-\varepsilon},Y_s^{-\varepsilon})$ such that 
\[
\exp(\psi_{-\varepsilon}(u_s(1))) = \exp(u_s(1)) \quad (\geq s-s_0-a).
\]

If $\ell$ is a Kolyvagin prime and $\lambda_L$ is a prime of $L$ above $\ell$, recall the map $\alpha_{\lambda_L,s}: \H^1(L,Y_s) \to Y_s(L)$ from \eqref{def-alpha-global}. Its restriction to each of the eigenspaces for the action of complex conjugation gives rise to maps $\alpha_{\lambda_L,s}^{\pm}: \H^1(L,Y_s)^{\pm} \to Y_s(L)^{\pm}$, and by a slight abuse of notation we still denote by $\alpha_{\lambda_L,s}^{\pm}$ the restrictions of these maps to $T^{\pm}$. The map $\alpha_{\lambda_L,s}$ corresponds to evaluation at $\Frob(\lambda_L)$, thus one can find a Kolyvagin prime $\ell$ such that $p^{s'}\mid \ell+1\pm a_{\ell}$, $p^{s'+a+1} \nmid \ell+1\pm a_{\ell}$ and $\alpha_{\lambda_L,s}^{\pm} = p^b \psi_{\pm}$.

Now let $t \in T^{\varepsilon}$ be arbitrary. By virtue of the reciprocity law,  
\begin{equation}\label{eq:RLoneprime}
\langle t_{\lambda}, p^{s_2}\kappa_s(\ell)_{\lambda}\rangle_{\lambda,s} = 0 \quad \text{in } \ZZ/p^s\ZZ,
\end{equation}
and the choice of $\ell$ together with Corollary \ref{cor:pairing} then imply that 
\[
[\alpha_{\lambda,s}(t_{\lambda}), p^{s_2+a+1}u_{\ell,\varepsilon_{\ell}}\alpha_{\lambda,s}(u_s(1)_{\lambda})]_{s} = 1.
\]
Now using that $t \in T^{\varepsilon}$, $u_s(1) \in T^{-\varepsilon}$, and the above relation between $\alpha_{\lambda_L,s}^{\pm}$ and $\psi_{\pm}$, it follows that
\[
[\psi_{\varepsilon}(t), p^{s_2+a+2b+1}u_{\ell,\varepsilon_{\ell}}\psi_{-\varepsilon}(u_s(1))]_{s} = 1.
\]
Since $[,]_s$ is non-degenerate, we infer that $p^{s_0+s_2+2a+2b+1}T^{\varepsilon} = 0$, and hence that 
\[
p^{s_0+s_2+3a+2b+1}(\Sel^{(s)}_{\wp}(f,K))^{\varepsilon} = 0.
\]

Next we look at the eigenspace $T^{-\varepsilon}$. As before, one can choose elements $\phi_{\pm} \in \Hom(T^{\pm},Y_s^{\pm})$ such that 
\[
\exp(\phi_{\varepsilon}(u_s(\ell))) = \exp(u_s(\ell))
\]
and
\[
\exp(\phi_{-\varepsilon} \text{ mod } \mathcal O_{\wp} \psi_{-\varepsilon}) = \exp(\Hom(T^{-\varepsilon},Y_s^{-\varepsilon})/\mathcal O_{\wp} \psi_{-\varepsilon}) = \exp(\ker(\psi_{-\varepsilon})).
\]
Notice that the choice of the prime $\ell$ implies that 
\begin{align*}
\exp(u_s(\ell)) & \geq \exp(\kappa_s(\ell)) - a \geq \exp(\kappa_s(\ell)_{\lambda}) - a \geq \exp(\kappa_s(1)_{\lambda}) - 2a \geq \\ 
& \geq \exp(p^b\psi_{-\varepsilon}(u_s(1))) -2a = \exp(u_s(1)) -2a-b \geq s-s_0-3a-b.
\end{align*}

As above, one can find a second Kolyvagin prime $\ell' \neq \ell$ such that $p^{s'}\mid \ell'+1\pm a_{\ell'}$, $p^{s'+a+1} \nmid \ell'+1\pm a_{\ell'}$ and $\alpha_{\lambda_L',s}^{\pm} = p^b \phi_{\pm}$. For $t \in \ker(\psi_{-\varepsilon}) \subseteq T^{-\varepsilon}$, the reciprocity law reads
\[
p^{s_2}\langle t_{\lambda}, \kappa_s(\ell\ell')_{\lambda}\rangle_{\lambda,s} + p^{s_2}\langle t_{\lambda'}, \kappa_s(\ell\ell')_{\lambda'}\rangle_{\lambda',s} = p^{s_2}\langle t_{\lambda'}, \kappa_s(\ell\ell')_{\lambda'}\rangle_{\lambda',s} = 0 \quad \text{in } \ZZ/p^s\ZZ,
\]
where the first term vanishes because of \eqref{eq:RLoneprime} and part iv) of Proposition \ref{localizationkappas}. This gets translated, thanks to Proposition \ref{prop:pairingsglobal} into the identity 
\[
[\phi_{-\varepsilon}(t), p^{2b+s_1+s_2+a+1}u_{\ell',\varepsilon_{\ell\ell'}}\phi_{\varepsilon}(u_s(\ell))]_s = 1.
\]
As a consequence, the kernel of $\psi_{-\varepsilon}: T^{-\varepsilon} \to Y_s^{-\varepsilon}$ is killed by $p^{s_0+s_1+s_2+4a+3b+1}$.

Finally, the assumption that $y_0$ is non-torsion in $\H^1(K,V_{\wp})$ implies the existence of an integer $s_0 \geq 0$ such that, modulo torsion, $y_0$ is divisible by $p^{s_0}$ in $\H^1(K,V_{\wp})$ but not by $p^{s_0+1}$.  For the class $u_s(1)$, this means that $u_s(1) = p^{s_0}x + t$ for some $x$, $t$ in the image of $\Phi$ with $p^{s_1}t = 0$, as the torsion part of $\H^1(K,V_{\wp})$ is killed by $p^{s_1}$. Thus for large enough $s$, the following relation holds
\[
\exp(\psi_{-\varepsilon}(x)) = \exp(x) \geq s - a.
\]
Besides, the map $\psi_{-\varepsilon}: T^{-\varepsilon} \to Y_s^{-\varepsilon}$ induces an exact sequence 
\[
0 \, \, \to \, \, \frac{\ker(\psi_{-\varepsilon})}{\ast} \, \, \to \, \, \frac{T^{-\varepsilon}}{\mathcal O_{\wp}t + \mathcal O_{\wp}x} \, \, \stackrel{\psi_{-\varepsilon}}{\to} \, \, \frac{Y_s^{-\varepsilon}}{\mathcal O_{\wp}\psi_{-\varepsilon}(t) + \mathcal O_{\wp}\psi_{-\varepsilon}(x)},
\]
in which the first term is killed by $p^{s_0+s_1+s_2+4a+3b+1}$ (because so is $\ker(\psi_{-\varepsilon})$), and the last term is killed by $p^a$. From this one concludes that
\[
(\Sel^{(s)}_{\wp}(f,K))^{-\varepsilon}/(\mathcal O_{\wp}t + \mathcal O_{\wp}x)
\]
is killed by $p^{s_0+s_1+s_2+6a+3b+1}$. As $s$ tends to $\infty$, one deduces that
\[
p^e \Sel^{(\infty)}_{\wp}(f,K)/((F_{\wp}/\mathcal O_{\wp})y_0) = 0
\]
for some $e$. Using that $\text{Im}(\Phi)$ is divisible in $\Sel^{(\infty)}_{\wp}(f,K)$, this identity proves our claim on $\text{Im}(\Phi)$, and shows that for a sufficiently large $s$, $(\Sh_{\wp^{\infty}})^{\varepsilon} = (\Sel^{(s)}_{\wp}(f,K))^{\varepsilon}$ and $(\Sel^{(s)}_{\wp}(f,K))^{-\varepsilon}/(\mathcal O_{\wp}t + \mathcal O_{\wp}x)$ surjects onto $(\Sh_{\wp^{\infty}})^{-\varepsilon}$. Hence the theorem is proved.
\end{proof}

\end{document}